 \documentclass{siamart190516}

\usepackage{upgreek} 
\usepackage{tikz}
\usetikzlibrary{fadings}
\usetikzlibrary{patterns}
\usetikzlibrary{shadows.blur}
\usetikzlibrary{shapes}
\usepackage{pgfplots}
\pgfplotsset{compat=1.14}

\newcommand{\cN}{\mathcal{N}}

\newcommand{\cH}{\mathcal{H}}
\newcommand{\cT}{\mathcal{T}}
\newcommand{\cE}{\mathcal{E}}

\newcommand{\bR}{\mathbb{R}}

\newcommand{\bN}{\mathbb{N}}

\newcommand{\rd}{\mathrm{d}}
\newcommand{\rN}{\mathrm{N}}

\newcommand{\sfh}{\mathsf{h}}
\newcommand{\sfb}{\mathsf{b}}

\newcommand{\dist}{\operatorname{dist}}

\usepackage{amsopn}

\usepackage{lipsum}
\usepackage{amsfonts}
\usepackage{graphicx}
\usepackage{epstopdf}
\usepackage{algorithmic}
\ifpdf
  \DeclareGraphicsExtensions{.eps,.pdf,.png,.jpg}
\else
  \DeclareGraphicsExtensions{.eps}
\fi

\newsiamremark{remark}{Remark}
\newsiamremark{hypothesis}{Hypothesis}
\crefname{hypothesis}{Hypothesis}{Hypotheses}
\newsiamthm{claim}{Claim}

\begin{document}
\headers{Edge Basis Functions with Exponential Convergence}{Y. Chen, T.Y. Hou, and Y. Wang}

\title{Exponential Convergence for Multiscale Linear Elliptic PDEs via Adaptive Edge Basis Functions\thanks{Submitted to the editors DATE: July 2020.
\funding{This research is in part supported by NSF Grants DMS-1912654 and DMS-1907977. Y. Chen is supported by the Caltech Kortchak Scholar Program.}}}

\author{Yifan Chen\thanks{Applied and Computational Mathematics, Caltech
  (\email{yifanc@caltech.edu}, \email{hou@cms.caltech.edu}).}
  \and Thomas Y. Hou\footnotemark[2]
\and Yixuan Wang\thanks{School of Mathematical Science, Peking University
  (\email{roywangyx@pku.edu.cn}).}}

\maketitle

\begin{abstract}
  In this paper, we introduce a multiscale framework based on adaptive edge basis functions to solve second-order linear elliptic PDEs with rough coefficients. One of the main results is that we prove the proposed multiscale method achieves nearly exponential convergence in the approximation error with respect to the computational degrees of freedom.  Our strategy is to perform an energy orthogonal decomposition of the solution space into a coarse scale component comprising $a$-harmonic functions in each element of the mesh, and a fine scale component named the bubble part that can be computed locally and efficiently. The coarse scale component depends entirely on function values on edges. Our approximation on each edge is made in the Lions-Magenes space $H_{00}^{1/2}(e)$, which we will demonstrate to be a natural and powerful choice. We construct edge basis functions using local oversampling and singular value decomposition.  When local information of the right-hand side is adaptively incorporated into the edge basis functions, we prove a nearly exponential convergence rate of the approximation error. Numerical experiments validate and extend our theoretical analysis; in particular, we observe no obvious degradation in accuracy for high-contrast media problems.
\end{abstract}


\tableofcontents
\section{Introduction}
        Multiscale methods have been widely deployed in the numerical simulation of Partial Differential Equations (PDEs). They provide an efficient way for modeling, representing, computing, and quantifying the solution at a variety of scales that are of interest in applications. A critical reason why multiscale methods are so powerful is that they allow different scales to be decoupled and treated in their own fashion, with the structured information fully exploited at each scale. Simultaneously, a seamless coupling scheme could be designed to combine different scales to get the final solution. This paper aims to study a particular coarse-fine scale decomposition of the solution space for solving the multiscale linear elliptic equation. The coarse scale component comprises $ a $-harmonic functions in each local element of the mesh and is approximated via edge basis functions. The fine scale component is solved by a local computation that is efficient and parallelizable. The combination of the two components leads to nearly exponential convergence in the approximation error with respect to the computational degrees of freedom; namely, it decays as $O\left(\exp(-Cm^\gamma)\right)$ for some $C,\gamma>0$, where $m$ is the number of degrees of freedom. This paper is a further development of the work \cite{hou2015optimal}.
    \subsection{Background and Context} Let $\Omega \subset \bR^d$ be a bounded, connected, open domain with a Lipschitz boundary $\partial \Omega$. Consider the linear elliptic equation: 
    \begin{equation}
    \left\{
    \begin{aligned}
    \label{eqn: elliptic rough}
    -\nabla \cdot (a \nabla u )&=f, \quad \text{in} \  \Omega\\
    u&=0, \quad \text{on} \  \partial\Omega\, .
    \end{aligned}
    \right.
    \end{equation}
    Here, we assume $u \in H_0^1(\Omega)$ and $f \in L^2(\Omega)$. The coefficients $a \in L^{\infty}(\Omega)$ and $0 < a_{\min} \leq a(x) \leq a_{\max}<\infty$ for all $x \in \Omega$. Typically, $a(x)$ has spatial oscillations, which leads to a multiscale behavior in the solution $u(x)$. No scale separation or periodicity of $a$ is assumed here. Our goal is to numerically solve the equation \eqref{eqn: elliptic rough}. 
    
    We adopt the standard Garlerkin methodology. Given a finite dimensional space $V_H \subset H_0^1(\Omega)$, 
    we can write a weak formulation of \eqref{eqn: elliptic rough} as
    \begin{equation}
    \label{eqn: weak formulation}
        \text{Find}\  u_H \in V_H \quad \text{such that} \quad \int_{\Omega} a \nabla u_H \cdot \nabla v = \int_{\Omega} fv \quad \text{for any } v \in V_H\, .
    \end{equation}
    Elements of $V_H$ are called \textit{basis functions}. The solution $u_H$ is a linear combination of basis functions, which approximates the true solution $u$. Denote the energy norm as $\|u\|_{H_a^1(\Omega)}:=\int_{\Omega} a|\nabla u|^2$. The standard finite element theory implies
    \begin{equation}
    \label{eqn: Galerkin projection}
        \|u-u_H\|_{H_a^1(\Omega)}=\inf_{v \in V_H} \|u-v\|_{H_a^1(\Omega)}\, .
    \end{equation}

    We can interpret \eqref{eqn: elliptic rough} as a function approximation problem. We are asked to approximate $u$ given the right-hand side data $f$. The Galerkin method \eqref{eqn: weak formulation} allows us to get a projection of $u$ into $V_H$, i.e., the best approximation of $u$ in $V_H$, using the data $f$. For the sake of accuracy, space $V_H$ should approximate the solution well in the energy norm, according to \eqref{eqn: Galerkin projection}. Because $a$ is rough, $u$ contains highly oscillatory patterns. Simple piecewise polynomials used in the conventional finite element method cannot lead to a satisfactory approximation in this general setting \cite{babuvska2000can}.
    \subsection{Coarse-Fine Decomposition of Solution Space}
    \label{subsec: Coarse-Fine Decomposition of Solution Space}
    For the reason mentioned at the end of the last subsection, we start to study the solution space of \eqref{eqn: elliptic rough}. For generality, we assume $f\in H^{-1}(\Omega)$ for now; this condition suffices for the existence of solution $u \in H_0^1(\Omega)$ in \eqref{eqn: elliptic rough}. In this case, the solution space is the whole $H_0^1(\Omega)$, because $-\nabla \cdot (a\nabla \cdot)$ leads to an isomorphism between $H_0^1(\Omega)$ and $H^{-1}(\Omega)$.
    \subsubsection{Decomposition of Solution Space}

    Let $\cT_H$ be a regular partition of the domain $ \Omega $ into finite elements such as triangles, quadrilaterals, and tetrahedra, with a mesh size $H$.
    In each element $T \in \cT_H$, the solution $u$ satisfies the elliptic equation in $T$, with a Dirichlet boundary condition on $\partial T$ determined by the values of $u$ on the boundary. We can locally decompose the solution as  $u=u_{T}^{\sfh}+u_T^{\sfb}$, where the two components satisfy respectively 
    \begin{equation}
    \label{eqn: harmonic-bubble splitting}
    \begin{aligned}
    &\left\{
    \begin{aligned}
    -\nabla \cdot (a \nabla u_T^\sfh )&=0, \quad \text{in} \  T\\
    u_T^\sfh&=u, \quad \text{on} \  \partial T\, ,
    \end{aligned}
    \right.
    \\
    &\left\{
    \begin{aligned}
    -\nabla \cdot (a \nabla u^\sfb_T )&=f, \quad \text{in} \  T\\
    u^\sfb_T&=0, \quad \text{on} \  \partial T\, .
    \end{aligned}
    \right.
    \end{aligned}
    \end{equation}
    The part $u^{\sfh}_T$ incorporates  the boundary value of $u$, while $u^{\sfb}_T$ contains information of the right-hand side. System \eqref{eqn: harmonic-bubble splitting} admits solutions because $u\in H^{1/2}(\partial T)$ when it is restricted to $\partial T$, and $f \in H^{-1}(T)$ when it is restricted to $T$.
    
    Furthermore, we can define a global function $u^{\sfh}$ and $u^{\sfb}$ in $\Omega$, such that $u^{\sfh}(x)=u^{\sfh}_T(x)$ and $u^{\sfb}(x)=u^{\sfb}_T(x)$ when $x \in T$ for each $T$. At the intersection of different $T \in \cT_H$,  $u_T^{\sfh}$ (resp. $u_T^{\sfb}$) has a unique value. Therefore, $u^\sfh$ and $u^\sfb$ are globally well defined. Moreover, they belong to $H_0^1(\Omega)$ by standard properties of Sobolev's space.
    
    The component $u^{\sfh}_T$ (resp. $u^{\sfh}$) is called the local (resp. global) \textit{$a$-harmonic part} and $u^{\sfb}_T$ (resp. $u^{\sfb}$) is the local (resp. global) \textit{bubble part}, of the solution $u$. The naming of the bubble part can be traced back to \cite{brezzi1994choosing, hughes1995multiscale, franca1995bubble, franca1996deriving}. By construction, $u^{\sfh}_T$ is orthogonal to $u^{\sfb}_T$ in $T$ under the energy norm; the same fact holds for $u^{\sfh}$ and $u^{\sfb}$ by splitting the inner product in $\Omega$ into a sum of local inner product.

    In this way, we get an energy orthogonal decomposition of the solution $u=u^\sfh+u^\sfb$, which further yields a decomposition of the solution space $H_0^1(\Omega)$. We write
    \begin{equation}
    \label{eqn: large, fine decomposition of solution space}
        \boxed{H_0^1(\Omega)=V^{\sfh}\oplus_a V^{\sf{b}}}
    \end{equation}such that $V^{\sfh}$ is the space of $u^{\sfh}$ and $V^{\sfb}$ of $u^{\sfb}$. More precisely, we have
    \begin{equation}
        \begin{aligned}
        &V^{\sfh}=\{v \in H_0^1(\Omega): -\nabla \cdot (a\nabla v)=0 \ \text{in every } T \in \cT_H  \}\, ,\\
        &V^{\sfb}=\{v \in H_0^1(\Omega): v=0 \ \text{on } \partial T, \ \text{for every } T \in \cT_H  \}\, .
        \end{aligned}
    \end{equation}
    The boundary values are understood in the sense of trace. The symbol $\oplus_a$ in \eqref{eqn: large, fine decomposition of solution space} denotes the direct sum that is energy orthogonal.
    \subsubsection{Coarse and Fine Scale Components}
    Now and hereafter in this paper, we assume $f \in L^2(\Omega)$. The solution space of \eqref{eqn: elliptic rough} for all such $f$ will be a subspace of $H_0^1(\Omega)$. We can still use the decomposition for the solution $u=u^\sfh+u^\sfb$.  In this case, we will understand $V^{\sfb}$ as the fine scale or microscopic space, because functions in $V^{\sfb}$ oscillate at a frequency larger than $1/H$, due to the classical elliptic estimate upon scaling for each element:
    \begin{equation}
    \label{eqn: bubble part, small}
        \|u^\sfb\|_{H_a^1(\Omega)}\leq CH\|f\|_{L^2(\Omega)}\, ,
    \end{equation}
    for some constant $C$ independent of $u$ and $H$. Conversely, we refer to $V^{\sfh}$ as the 
   coarse scale or macroscopic space. Thus, we get a coarse-fine scale decomposition of the solution space.
    
    Let us return to problem \eqref{eqn: elliptic rough}: we want to approximate a function $u \in V^{\sfh}\oplus_a V^{\sfb}$, given the right-hand side data $f \in L^2(\Omega)$. With the decomposition \eqref{eqn: large, fine decomposition of solution space}, we will treat the two scales differently. By definition, the fine scale part $u^{\sfb}\in V^{\sfb}$ depends locally on $f$, and thus can be computed efficiently by solving local elliptic equations in each $T \in \cT_H$. Moreover, since this part has a small energy norm \eqref{eqn: bubble part, small}, it can be completely ignored if we only need $O(H)$ accuracy in this norm. Nevertheless, we emphasize that the efficient computation of the fine scale part is the key to nearly exponential convergence that we will establish in this paper.
    
    For the coarse scale $a$-harmonic part $u^{\sfh} \in V^{\sfh}$, we invoke the Galerkin framework \eqref{eqn: weak formulation}, while now we seek a $V_H$ that is a subspace of $V^{\sfh}$.  In this case, we can get a more refined version of \eqref{eqn: Galerkin projection} due to the orthogonality:
    \begin{equation}
        \label{eqn: refined Galerkin projection}
        \|u^{\sfh}-u_H\|_{H_a^1(\Omega)}=\inf_{v \in V_H} \|u^{\sfh}-v\|_{H_a^1(\Omega)}\, .
    \end{equation}
    Thus, for our choice of $V_H \subset V^{\sfh}$, the Galerkin solution $u_H$ does not affect the approximation of $u^{\sfb}$. The computations of $u^{\sfh}$ and $u^{\sfb}$ are completely decoupled. 
    
  Therefore, the problem reduces to identifying a good approximation space $V_H$ for the coarse scale component $V^{\sfh}$. By definition, any function in $V^{\sfh}$ can be entirely characterized by its value on edges; that is, $V^{\sfh}$ could essentially be treated as a function space on edges. Thus, it calls for a systematic design of \textit{edge basis functions}, whose $a$-harmonic extension to each element will constitute the desired approximation space $V_H$. 
  
    \subsection{Our Contributions}
    For simplicity of presentation, we mainly focus on the case $d=2$; remarks on how to generalize to $d\geq 3$ will be made in Section \ref{sec: discussion}.
    
    First, we develop a framework for constructing edge basis functions that can lead to rigorous error estimates. The construction is similar to that in \cite{hou2015optimal}, and there are two parts, the interpolation part and enrichment part. The interpolation part comprises nodal basis functions. For the enrichment part, we propose to use the Lions-Magenes space $H_{00}^{1/2}(e)$ 
    as the natural function space on each edge $e$. An appropriate norm for measuring the approximation error is studied, and a systematic way of coupling local approximation accuracy to global accuracy is established.
    
    Then, we discuss how to achieve local approximation accuracy.  A general strategy is illustrated, based on oversampling and singular value decomposition; the approach in \cite{hou2015optimal} falls into our framework. Theoretically, we prove the exponential decay of singular values for some restriction operator mapping an $a$-harmonic function to its interpolation residue on edges. This fact, combined with the efficient computation of an oversampling bubble part, leads to the nearly exponentially decaying approximation error with respect to the degrees of freedom. If the oversampling bubble part is not computed, i.e., the information of $f$ is not incorporated into the edge basis functions, then we will get $O(H)$ approximation accuracy in the energy norm by using at most $O(\log^{d+1}(1/H))$ degrees of freedom; this matches the state-of-the-art result for linear elliptic PDEs with rough coefficients in the same setting.
    
    Finally, we present numerical experiments to validate our theoretical analysis. They match the prediction and achieve nearly exponential convergence. Moreover, since every step in our algorithm is $a(x)$-adapted, we expect the method to be robust with respect to high contrast in $a(x)$. Indeed, our experiment demonstrates no obvious degradation in accuracy for high contrast media problems.
    \subsection{Related Works}
    From one perspective, the method in this paper can be understood as a member of the family of Multiscale Finite Element Method (MsFEM) \cite{hou_multiscale_1997, hou1999convergence, efendiev2000convergence}, Generalized Multiscale Finite Element Methods (GMsFEM) \cite{efendiev_generalized_2013} and Generalized Finite Element Methods (GFEM) \cite{babuvska1983generalized}. This family of methods usually starts with a domain decomposition of $\Omega$, then builds up local approximation spaces that can capture the multiscale behavior, and finally couple these local spaces to a global approximation space $V_H$. There are two popular coupling schemes related to this paper. The first one is the partition of unity method (PUM) \cite{melenk1996partition}, introduced in the context of the GFEM.  In \cite{babuska2011optimal}, based on PUM, an optimal local basis is constructed for elliptic equations with rough coefficients, which achieves $O(H)$ accuracy in the energy norm using only $O(\log^{d+1} (1/H))$ number of basis functions in each local domain. Moreover, by using the right-hand side information, a nearly exponentially decaying error with respect to the number of basis functions can be achieved; see also \cite{babuvska2014machine, babuvska2020multiscale}. We will borrow some techniques in \cite{babuska2011optimal} to prove the nearly exponential convergence of our method.
    
    Another coupling scheme is the edge coupling, which originates from MsFEM. This paper aims to enrich the space of edge basis functions in MsFEM, to improve the approximation when the coefficients are rough. In the literature, there have been a number of works on enriching the edge basis functions.  Most of them either resort to non-conformal finite element, discontinuous Galerkin coupling, or use an additional partition of unity to form the global approximation space; see \cite{efendiev2009multiscale} and the references therein. One distinct difference between our work and other works is that no partition of unity is required for constructing our basis functions of enrichment. Moreover, approximation in a novel space $H_{00}^{1/2}(e)$ is employed for each edge $e$, which leads to the desired guarantee of accuracy. We remark that some of the recent works \cite{fu2019edge, li2019convergence} also consider enriching edge basis functions and provide a rigorous theoretical guarantee. They use $L^2(e)$ space for the edge, which corresponds to a weaker norm compared to $H_{00}^{1/2}(e)$. To the best of our knowledge, this choice does not directly lead to the state-of-the-art result that $O(\log^{d+1}(1/H))$ amount of basis functions suffices for $O(H)$ accuracy in the energy norm.
   
   On the other hand, we can also interpret the method in the framework of Variational Multiscale Methods (VMS) \cite{hughes_variational_1998}, since we decompose the solution space into coarse and fine scale components via the harmonic-bubble splitting. The decomposition, as it appears, is different from the traditional decomposition in VMS. In the literature, the decomposition can be done in a $L^2$ orthogonal sense, for example, in the projection method and wavelet homogenization \cite{dorobantu1998wavelet, engquist2002wavelet}. Furthermore, energy decomposition has been rigorously established, such as the Local Orthogonal Decomposition (LOD) \cite{malqvist_localization_2014, kornhuber2018analysis} and its multiresolution generalization Gamblets \cite{owhadi_multigrid_2017, owhadi2019operator}; see also discussions in the flux norm approach \cite{berlyand2010flux, owhadi2011localized} and the rough polyharmonic splines \cite{owhadi2014polyharmonic}. Our method uses an energy orthogonal decomposition as well, while it is different from that in the LOD approach. The coarse space in our method depends entirely on the edge values of functions. Localization of the coarse part is established directly through a prescribed oversampling domain and a corresponding singular value decomposition. A key merit of our approach is that the fine scale part can be computed very efficiently, which could explain why the nearly exponential convergence can be rigorously obtained.
    
    \subsection{Organization of this Paper}
   The rest of the paper is organized as follows. In Section \ref{sec: Localized Edge Basis Functions for The Coarse Space}, we discuss how to localize the approximation of the coarse component onto every edge. An appropriate space $H_{00}^{1/2}(e)$ for each edge $e$ is introduced and its property is rigorously established. Using this space, local to global error estimates can be derived. Section \ref{sec: Local Approximation via Oversampling} is devoted to studying how to achieve the local approximation. Oversampling and singular value decomposition are introduced to achieve this goal, and the nearly exponential convergence is rigorously proved. 
    Numerical experiments are performed in Section \ref{sec: Numerical experiments} to demonstrate the effectiveness of our proposed approach. In Section \ref{sec: discussion}, we summarize and conclude this paper, together with several discussions on potential generalizations of this method.
    \section{Localized Edge Basis Functions for Coarse Space}
    \label{sec: Localized Edge Basis Functions for The Coarse Space}
    Following the idea in Subsection \ref{subsec: Coarse-Fine Decomposition of Solution Space}, we need to find a good approximation space $V_H$ for the coarse scale component $V^{\sfh}$. In this section, we will explain in detail how to localize the approximation task via edge basis functions. Specifically, we study how the approximation error can be localized to every single edge and how to integrate these local errors to a global accuracy guarantee. We will elaborate on the topic of getting desired local approximation in Section \ref{sec: Local Approximation via Oversampling}.

    \subsection{Mesh, Geometry and Notation} To begin with, we present some notations on the mesh structure of the domain $\Omega \subset \bR^d$. We focus on $d=2$, and will provide remarks on generalization to $d\geq 3$ in Section \ref{sec: discussion}. We have nodes, edges, and elements in the mesh. Their neighbourship is very useful for describing the geometry.
    \subsubsection{Elements}
    \label{subsec:elements}
    As in Subsection \ref{subsec: Coarse-Fine Decomposition of Solution Space}, we consider a shape regular partition of the domain $ \Omega $ into finite elements, such as triangles and quadrilaterals. The collection of elements is denoted by $\cT_H=\{T_1,T_2,...,T_r\}$; we adopt the convention that each $T$ is an open set. 
    The mesh size is $H$, i.e., $\max_{T \in \cT_H} \operatorname{diam}(T)=H$. We also assume that the mesh is uniform, i.e., $\min_{T \in \cT_H} \operatorname{diam}(T)\geq c_0 H$ for some $0<c_0\leq 1$ that is independent of $H$ and $T$.  The uniformity is used to simplify the discussion of the oversampling domain in Section \ref{sec: Local Approximation via Oversampling}. The result can be directly generalized to non-uniform mesh.
    
    The shape regular property implies there is a constant $c_1>0$ independent of $H$ and $T$, such that
    \begin{equation}
        \max_{T \in \cT_H} \frac{\operatorname{diam}(T)^d}{|T|} \leq c_1\, ,
    \end{equation}
    where, $|T|$ is the volume of $T$. In the theoretical analysis, we often need to re-scale the element to a standard reference element of diameter $O(1)$. The shape regular condition ensures that the distortion of the volume is bounded.

\subsubsection{Nodes, Edges and Their Neighbors} We let $\cN_H=\{x_1,x_2,...,x_p\}$ be the collection of interior nodes, and $\cE_H=\{e_1,e_2,...,e_q\}$ be the collection of edges except those fully on the boundary of $\Omega$. An edge $e \in \cE_H$ is defined such that there exists two different elements $T_i,T_j$ with $e=\overline{T}_i\bigcap\overline{T}_j$ that has co-dimension $1$ in $\bR^d$. We will use $E_H=\bigcup_{e \in \cE_H} e \subset \Omega$ to denote the whole edges as a set. The readers should not confuse $\cE_H$ with $E_H$.

The neighbourship of these nodes, edges and elements is an important part of the geometry of the mesh. We use the symbol $\sim$ to describe the neighbourship. More precisely, if we consider a node $x \in \cN_H$, an edge $e \in \cE_H$, and an element $T \in \cT_H$, then, (1) $x \sim e$ denotes $x \in e$; (2) $e \sim T$ denotes $e \subset \overline{T}$; (3) $x \sim T$ denotes $x \in \overline{T}$. The relationship $\sim$ is symmetric.

We use $\rN(\cdot,\cdot)$ to describe the union of neighbors as a set. For example, 
$\rN(x,\cE_H)=\bigcup \{e \in \cE_H: e \sim x\} \subset E_H$, and $\rN(x,\cT_H)=\bigcup \{T \in \cT_H: T \sim x\} \subset \Omega$. Also, we have $\rN(e,\cT_H)=\bigcup \{T \in \cT_H: T \sim e\} \subset \Omega$.

The geometric relationships are illustrated in Figure \ref{fig:mesh geometry}.
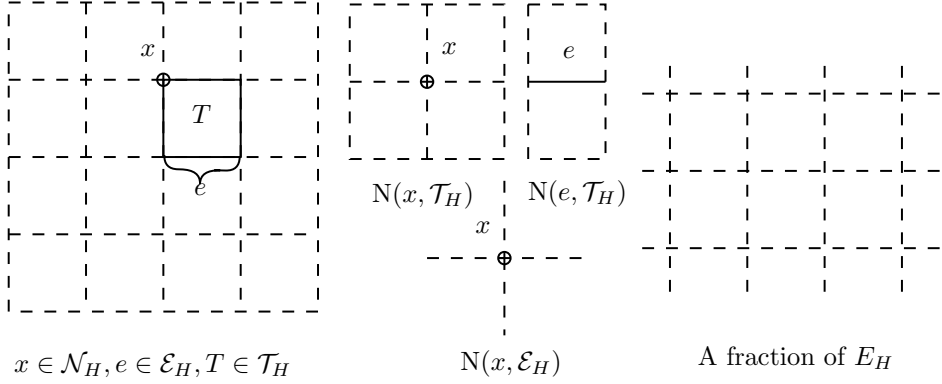
\begin{figure}[t]
    \centering
\tikzset{every picture/.style={line width=0.75pt}} 

\begin{tikzpicture}[x=0.75pt,y=0.75pt,yscale=-1,xscale=1]

\draw  [draw opacity=0][dash pattern={on 4.5pt off 4.5pt}] (45,47) -- (201,47) -- (201,203) -- (45,203) -- cycle ; \draw  [dash pattern={on 4.5pt off 4.5pt}] (84,47) -- (84,203)(123,47) -- (123,203)(162,47) -- (162,203) ; \draw  [dash pattern={on 4.5pt off 4.5pt}] (45,86) -- (201,86)(45,125) -- (201,125)(45,164) -- (201,164) ; \draw  [dash pattern={on 4.5pt off 4.5pt}] (45,47) -- (201,47) -- (201,203) -- (45,203) -- cycle ;
\draw   (123,86) -- (162,86) -- (162,125) -- (123,125) -- cycle ;
\draw   (123.5,124) .. controls (123.38,128.67) and (125.65,131.06) .. (130.31,131.18) -- (131.68,131.22) .. controls (138.34,131.39) and (141.61,133.81) .. (141.49,138.48) .. controls (141.61,133.81) and (145,131.57) .. (151.67,131.74)(148.67,131.67) -- (154.32,131.81) .. controls (158.99,131.94) and (161.38,129.67) .. (161.5,125) ;
\draw   (126.13,86) .. controls (126.13,84.27) and (124.73,82.88) .. (123,82.88) .. controls (121.27,82.88) and (119.88,84.27) .. (119.88,86) .. controls (119.88,87.73) and (121.27,89.13) .. (123,89.13) .. controls (124.73,89.13) and (126.13,87.73) .. (126.13,86) -- cycle ;
\draw  [draw opacity=0][dash pattern={on 4.5pt off 4.5pt}] (364.5,79) -- (516.5,79) -- (516.5,199) -- (364.5,199) -- cycle ; \draw  [dash pattern={on 4.5pt off 4.5pt}] (378.5,79) -- (378.5,199)(417.5,79) -- (417.5,199)(456.5,79) -- (456.5,199)(495.5,79) -- (495.5,199) ; \draw  [dash pattern={on 4.5pt off 4.5pt}] (364.5,93) -- (516.5,93)(364.5,132) -- (516.5,132)(364.5,171) -- (516.5,171) ; \draw  [dash pattern={on 4.5pt off 4.5pt}]  ;

\draw   (259.13,87) .. controls (259.13,85.27) and (257.73,83.88) .. (256,83.88) .. controls (254.27,83.88) and (252.88,85.27) .. (252.88,87) .. controls (252.88,88.73) and (254.27,90.13) .. (256,90.13) .. controls (257.73,90.13) and (259.13,88.73) .. (259.13,87) -- cycle ;
\draw  [draw opacity=0][dash pattern={on 4.5pt off 4.5pt}] (217,48) -- (295,48) -- (295,126) -- (217,126) -- cycle ; \draw  [dash pattern={on 4.5pt off 4.5pt}] (256,48) -- (256,126) ; \draw  [dash pattern={on 4.5pt off 4.5pt}] (217,87) -- (295,87) ; \draw  [dash pattern={on 4.5pt off 4.5pt}] (217,48) -- (295,48) -- (295,126) -- (217,126) -- cycle ;

\draw   (298.13,176) .. controls (298.13,174.27) and (296.73,172.88) .. (295,172.88) .. controls (293.27,172.88) and (291.88,174.27) .. (291.88,176) .. controls (291.88,177.73) and (293.27,179.13) .. (295,179.13) .. controls (296.73,179.13) and (298.13,177.73) .. (298.13,176) -- cycle ;
\draw  [dash pattern={on 4.5pt off 4.5pt}]  (291.88,176) -- (334,176) ;
\draw  [dash pattern={on 4.5pt off 4.5pt}]  (295,137) -- (295,176) ;
\draw  [dash pattern={on 4.5pt off 4.5pt}]  (256,176) -- (295,176) ;
\draw  [dash pattern={on 4.5pt off 4.5pt}]  (295,176) -- (295,215) ;

\draw [line width=0.75]    (307,87) -- (346,87) ;
\draw  [draw opacity=0][dash pattern={on 4.5pt off 4.5pt}] (307,48) -- (346,48) -- (346,126) -- (307,126) -- cycle ; \draw  [dash pattern={on 4.5pt off 4.5pt}]  ; \draw  [dash pattern={on 4.5pt off 4.5pt}]  ; \draw  [dash pattern={on 4.5pt off 4.5pt}] (307,48) -- (346,48) -- (346,126) -- (307,126) -- cycle ;

\draw (136,97) node [anchor=north west][inner sep=0.75pt]   [align=left] {$\displaystyle T$};
\draw (137,137) node [anchor=north west][inner sep=0.75pt]   [align=left] {$\displaystyle e$};
\draw (110,67) node [anchor=north west][inner sep=0.75pt]   [align=left] {$\displaystyle x$};
\draw (33,223) node [anchor=north west][inner sep=0.75pt]   [align=left] { $\displaystyle \quad x\in \mathcal{N}_{H} ,e\in \mathcal{E}_{H} ,T\in \mathcal{T}_{H}$};
\draw (262,66) node [anchor=north west][inner sep=0.75pt]   [align=left] {$\displaystyle x$};
\draw (279,156) node [anchor=north west][inner sep=0.75pt]   [align=left] {$\displaystyle x$};
\draw (323,68) node [anchor=north west][inner sep=0.75pt]   [align=left] {$\displaystyle e$};
\draw (218,136) node [anchor=north west][inner sep=0.75pt]   [align=left] {$\displaystyle\ \  \rN( x,\mathcal{T}_{H})$};
\draw (301,135) node [anchor=north west][inner sep=0.75pt]   [align=left] {$\displaystyle \ \rN( e,\mathcal{T}_{H})$};
\draw (267,221) node [anchor=north west][inner sep=0.75pt]   [align=left] {$\displaystyle \ \rN( x,\mathcal{E}_{H})$};
\draw (391,219) node [anchor=north west][inner sep=0.75pt]   [align=left] {A fraction of $\displaystyle E_{H}$};

\end{tikzpicture}
    \caption{Geometry of the mesh}
    \label{fig:mesh geometry}
\end{figure}
 \subsubsection{Notation}
    \label{subsec: notation}
    We use the term ``edge basis function" to denote a function on $E_H$, and ``basis function" usually refers to a function in the full dimensional domain $\Omega$. When there is no ambiguity, we will write $\tilde{\psi}$ to denote a function defined on edges, while for $\psi$, we refer to as a function in full-dimensional domains. We use $\psi|_e$ to denote the function $\psi$ restricted to the set $e$.
    \subsection{Edge Approximation: Set-up} 
    Now, we proceed with the discussion in Subsection \ref{subsec: Coarse-Fine Decomposition of Solution Space}. Functions in the coarse space $V^{\sfh}$ depend entirely on their values on edges. Let us define the following space on edges: 
    \[\tilde{V}^{\sfh}:=\{\tilde{\psi}: E_H \to \bR, \text{ there exists a function } \psi \in V^{\sfh}, \text{ such that } \tilde{\psi}=\psi|_{E_H} \}\, .\] 
    We have $\tilde{V}^{\sfh}=H^{1/2}(E_H)$ by the trace theorem of the Sobolev space. There is a one-to-one correspondence between functions in $\tilde{V}^{\sfh}$ and $V^{\sfh}$: $\tilde{\psi} \in \tilde{V}^{\sfh} \leftrightarrow \psi \in V^{\sfh}$; namely, the following equation holds in each $T \in \cT_H$:
    \begin{equation}
    \label{eqn: edge bulk correspondence} 
    \left\{
    \begin{aligned}
    -\nabla \cdot (a \nabla \psi)&=0, \quad \text{in} \  T\\
    \psi&=\tilde{\psi}, \quad \text{on} \  \partial T\, .
    \end{aligned}
    \right.
    \end{equation}
     Using this correspondence, the coarse scale component on edges is denoted by $\tilde{u}^{\sfh} \in \tilde{V}^{\sfh}$. The approximation space for $\tilde{u}^{\sfh}$ is $\tilde{V}_H$, corresponding to $V_H$ for $u^\sfh$. 
    
    How shall we design $\tilde{V}_H$? 
    Computationally, we prefer local edge basis functions. Here, our idea of localized construction follows that proposed in \cite{hou2015optimal}. The first step is to use some local nodal basis functions to interpolate $\tilde{u}^{\sfh}$. Then, the interpolation residue can be localized to each edge, where more enrichment edge basis functions are designed for further approximation.
    
    \subsection{Interpolation Part}
    \label{subsec: interpolation part}
    We begin with the interpolation part.
    For each node $x_i \in \cN_H$, the nodal edge basis function $\tilde{\psi}_{i}$ satisfies $\tilde{\psi}_i(x_j)=\delta_{ij}$ for every $x_j \in \cN_H$, and $\tilde{\psi}_i(x)$ is supported on $\rN(x_i,\cE_H)$. The corresponding part of $\tilde{\psi}_{i}$ in $V^{\sfh}$ is $\psi_i(x)$, which is supported in the closure of $\rN(x_i,\cT_H)$. In this paper, we set $\tilde{\psi}_i$ to be the linear tent function used in \cite{hou_multiscale_1997}. These nodal basis functions constitute the interpolation part.  More general constructions   can be considered, for which we refer to the discussions in \cite{hou2015optimal}.
    
    With the interpolation part defined, we can introduce a nodal interpolation operator $I_H: \tilde{V}^{\sfh} \cap C(E_H) \to \tilde{V}^{\sfh} \cap C(E_H)$ such that \[I_H \tilde{v}:=\sum_{x_i \in \cN_H} \tilde{v}(x_i)\tilde{\psi}_i(x)\, ,\]  for any $\tilde{v} \in \tilde{V}^{\sfh} \cap C(E_H)$. We will also identify this operator as the mapping from $V^{\sfh} \cap C(\Omega)$ to $V^{\sfh} \cap C(\Omega)$, based on the correspondence between $V^{\sfh}$ and $\tilde{V}^{\sfh}$. That is, we will write $I_H v=\sum_{x_i \in \cN_H} v(x_i)\psi_i(x)$ for any $v \in V^{\sfh} \cap C(\Omega)$. We note that in this definition, the pointwise value is well defined for continuous functions.

    The nodal edge basis functions allow us to approximate $\tilde{u}^{\sfh}$ via interpolation. Since $f \in L^2(\Omega)$, a classical result from elliptic PDEs implies that the solution $u \in C^{\alpha}(\Omega)$ for some $0<\alpha<1$ that depends on the coefficients $a$ and the domain $\Omega$; for details see Theorems 8.22 and 8.29 in \cite{gilbarg2015elliptic}. Thus, $\tilde{u}^{\sfh} \in \tilde{V}^{\sfh}\cap C^{\alpha}(E_H)$, and we can apply the interpolation operator on it. The interpolation residue is $\tilde{u}^{\sfh}-I_H\tilde{u}^{\sfh}$. The enrichment part is introduced further to approximate the residue.
   \begin{remark}
    The interpolation part here is the same as the basis functions in MsFEM \cite{hou_multiscale_1997, hou1999convergence, efendiev2000convergence}. Our enrichment part to be presented in the next subsection is used to further improve the MsFEM so that it can handle rough coefficients with guaranteed accuracy.
    \end{remark}
    \subsection{Enrichment Part}
    \label{subsec: enrichment part}
    The residue $\tilde{u}^{\sfh}-I_H\tilde{u}^{\sfh}$ vanishes at nodal points.
    This property inspires us to localize the residue onto each edge.
    To illustrate the idea, let us fix one edge $e$ for now, and use $P_e$ to denote the restriction operator onto $e$, such that the function $P_e(\tilde{u}^{\sfh}-I_H\tilde{u}^{\sfh}):=(\tilde{u}^{\sfh}-I_H\tilde{u}^{\sfh})|_e$.  To approximate $P_e(\tilde{u}^{\sfh}-I_H\tilde{u}^{\sfh})$, we shall first identify a proper function space that it belongs to.
    
    Clearly, $P_e(\tilde{u}^{\sfh}-I_H\tilde{u}^{\sfh}) \in H^{1/2}(e)\bigcap C^{\alpha}(e)$, and it has value $0$ at the boundary of $e$. Based on this property, we can show $P_e(\tilde{u}^{\sfh}-I_H\tilde{u}^{\sfh}) \in H_{00}^{1/2}(e)$, the Lions-Magenes space on $e$; see Proposition \ref{prop: C alpha interpolation residue}.
    \begin{proposition}
    \label{prop: C alpha interpolation residue}
    Suppose $u$ satisfies equation \eqref{eqn: elliptic rough} for $d\leq 3$ and $f \in L^2(\Omega)$, then, we have $P_e(\tilde{u}^{\sfh}-I_H\tilde{u}^{\sfh}) \in H_{00}^{1/2}(e)$ for every $e \in \cE_H$.
    \end{proposition}
    \begin{proof}
      First, we recall the definition of the space $H_{00}^{1/2}(e)$  (Chapter 33 of \cite{tartar2007introduction}): it is the space of functions $v\in H^{1/2}(e)$ such that 
      \[\frac{v(x)}{\dist(x,\partial e)} \in L^2(e)\, , \]
      where $\dist(x,\partial e)$ is the Euclidean distance from $x$ to the boundary of $e$. Thus, we can define the $H_{00}^{1/2}(e)$ norm to be
      \begin{equation}
      \label{eqn: def H 00 1/2 norm}
        \|v\|_{H_{00}^{1/2}(e)}^2:=\int_e|v(x)|^2\, \rd x+\int_e\int_e\frac{|v(x)-v(y)|^2}{|x-y|^2}\, \rd x\rd y+\int_{e} \frac{|v(x)|^2}{\dist(x,\partial e)}\, \rd x<\infty\, ,
    \end{equation}
    and the Lions-Magenes space on $e$ consists functions with a finite $H_{00}^{1/2}(e)$ norm. We will show that any function $v$ on $e$ belonging to $H^{1/2}(e)\bigcap C^{\alpha}(e)$ and vanishing at $\partial e$ will be in the space $H_{00}^{1/2}(e)$. To see that, it suffices to show
    \[\int_{e} \frac{|v(x)|^2}{\dist(x,\partial e)}\, \rd x<\infty\, , \]
    because the first two terms in \eqref{eqn: def H 00 1/2 norm} are finite due to $v \in H^{1/2}(e)$. Without loss of generality, we work on $e=[0,1]$ (otherwise we can reparametrize the edge), then it follows that
    \begin{align*}
        \int_e \frac{|v(x)|^2}{\dist(x,\partial e)}\,\rd x &=\int_0^{1/2}\frac{|v(x)-v(0)|^2}{|x|}\,\rd x+\int_{1/2}^1 \frac{|v(x)-v(1)|^2}{|x-1|}\, \rd x\\
        &\leq C\left(\int_0^{1/2}|x|^{2\alpha-1}\, \rd x +\int_{1/2}^1 |x-1|^{2\alpha-1}\, \rd x \right) <\infty\, ,
    \end{align*}
    where $C$ is a constant such that $|v(x)-v(y)|\leq C|x-y|^{\alpha}$ due to $v \in C^{\alpha}(e)$.
    Thus, $v \in H_{00}^{1/2}(e)$. Taking $v=P_e(\tilde{u}^{\sfh}-I_H\tilde{u}^{\sfh})$ completes the proof.
    \end{proof}
     \begin{remark}
     \label{rmk: H 00 1/2 equivalent to 0 extension H 1/2}
     According to Chapter 33 of \cite{tartar2007introduction}, $H_{00}^{1/2}(e)$ can also be characterized as the space of functions in $H^{1/2}(e)$, such that their zero extensions to $E_H$ is still in $H^{1/2}(E_H)$. This is the key and in fact the only property that we will use for $H_{00}^{1/2}(e)$. 
     We need to do zero extension often, in
order to move from local approximation to global approximation.
     \end{remark}

     We will choose our enrichment part of edge basis functions in $H_{00}^{1/2}(e)$. Before we find an enrichment part $\tilde{v}_e \in H_{00}^{1/2}(e)$ to approximate $P_e(\tilde{u}^{\sfh}-I_H\tilde{u}^{\sfh})$, we need to understand first in which norm this approximation should occur. Because our final goal is to approximate $u^\sfh$ in the global energy norm, it is natural to use a local norm on $e$ whose connection to this energy norm could be established. It motivates the following definition that is originally proposed in \cite{hou2015optimal}. 
    
    Let $\tilde{\psi} \in H_{00}^{1/2}(e)$. We also write $\tilde{\psi}$ for its zero extension to $E_H$, which is in $H^{1/2}(E_H)$ by the definition of $H_{00}^{1/2}(e)$. Denote the $a$-harmonic extension of $\tilde{\psi}$ to $\Omega$ by $\psi$, which satisfies \eqref{eqn: edge bulk correspondence} for each $T \in \cT_H$. The support of $\psi$ is the closure of $\rN(e,\cT_H)$. Then, the definition of the norm is give below.
    \begin{definition}
    \label{def: H 1/2 norm}
      The $\cH^{1/2}(e)$ norm of $\tilde{\psi}\in H_{00}^{1/2}(e)$ is defined as:
    \begin{equation}
    \|\tilde{\psi}\|_{\cH^{1/2}(e)}^2:=\int_{\Omega}a|\nabla \psi |^2\, .
    \end{equation}
    For each element $T \sim e$, the $\cH^{1/2}_T(e)$ norm is defined as:
    \begin{equation}
        \|\tilde{\psi}\|_{\cH^{1/2}_T(e)}^2:=\int_{T}a|\nabla \psi |^2\, .
    \end{equation}
    \end{definition}
    That is, the $\cH^{1/2}(e)$ norm of $\tilde{\psi} \in H_{00}^{1/2}(e)$ is defined through the energy norm of its $a$-harmonic extension $\psi \in H^1_0(\Omega)$. Intuitively, if functions on edges can be approximated well in the $\cH^{1/2}(e)$ norm, then it is natural to expect that their $a$-harmonic extensions can also be approximated well. Theorem \ref{thm: Edge coupling error estimate} in the next subsection establishes this intuition in a rigorous way.
    
    To this end, we have to show that Definition \ref{def: H 1/2 norm} makes sense, i.e., this norm is well defined. In fact, we can prove that the $\cH^{1/2}(e)$ norm is equivalent to the $H_{00}^{1/2}(e)$ norm; see Proposition \ref{prop: equivalence of different edge norms}. 
     \begin{proposition}
    \label{prop: equivalence of different edge norms}
    For each edge $e \in \cE_H$, the $\cH^{1/2}(e)$ norm and the $H_{00}^{1/2}(e)$ norm are equivalent, up to a constant independent of the mesh size $H$.
    \end{proposition}
     \begin{proof}
     Let $\tilde{\psi} \in H_{00}^{1/2}(e)$. We also write $\tilde{\psi}$ for its zero extension to $E_H$, which is in $H^{1/2}(E_H)$ by the definition of $H_{00}^{1/2}(e)$.
     As in Remark \ref{rmk: H 00 1/2 equivalent to 0 extension H 1/2} and Chapter 33 of \cite{tartar2007introduction}, the $H_{00}^{1/2}(e)$ norm of $\tilde{\psi}$ on $e$ is equivalent to the $H^{1/2}(\partial T)$ norm of $\tilde{\psi}$ on $\partial T$, for any $T\sim e$. Thus, it suffices to show that the $H^{1/2}(\partial T)$ norm and the $\cH^{1/2}_T(e)$ norm are equivalent up to a constant independent of the mesh size $H$.
     
    First, we have a variational characterization for the $\cH^{1/2}_T(e)$ norm. Define the space of $H^1$ functions in $T$ with a boundary value $\tilde{\psi}$ as 
    \[V_{\tilde{\psi}}:=\{v\in H^1(T): v|_{\partial T}=\tilde{\psi}\}\, .\]
    Then, by a simple calculus of variation, we get
    \begin{equation}
        \|\tilde{\psi}\|_{\cH^{1/2}_T(e)}^2=\inf_{v \in V_{\tilde{\psi}}} \int_T a|\nabla v|^2\, .
    \end{equation}
    Thus, using $a_{\min}\leq a \leq a_{\max}$, we arrive at
    \begin{equation}
    \label{eqn: equivalent H 1/2 T and a harmonic extension norm}
        a_{\min}\inf_{v \in V_{\tilde{\psi}}} \int_T |\nabla v|^2 \leq \|\tilde{\psi}\|_{\cH^{1/2}_T(e)}^2 \leq a_{\max} \inf_{v \in V_{\tilde{\psi}}} \int_T |\nabla v|^2\, .
    \end{equation}
    On the other hand, the trace theorem implies that there exists some constant $C_1,C_2$ independent of $H$ such that
    \begin{equation}
    \label{eqn: H 1/2 and harmonic extension norm}
        C_1\inf_{v \in V_{\tilde{\psi}}} \int_T |\nabla v|^2\leq \|\tilde{\psi}\|_{H^{1/2}(\partial T)}^2\leq C_2\inf_{v \in V_{\tilde{\psi}}} \int_T |\nabla v|^2\, .
    \end{equation}
    Combining \eqref{eqn: equivalent H 1/2 T and a harmonic extension norm} and \eqref{eqn: H 1/2 and harmonic extension norm} completes the proof.
    \end{proof}
    
    In this way, we can understand $\cH^{1/2}(e)$ as a variant of the $H_{00}^{1/2}(e)$ norm, while it can additionally incorporate the information of $a$.  The next subsection will build the theory of combining localized approximation in the $\cH^{1/2}(e)$ norm to a global guarantee of accuracy. This theory, in turn, will demonstrate that $\cH^{1/2}(e)$ is the appropriate norm to use for measuring approximation errors.
    \subsection{Integrating Local Error to Global Error} In this subsection, 
    we show that error estimates in the $\cH^{1/2}(e)$ norm on edges can be directly connected to the global approximation error to $u^\sfh$ in the energy norm; see the following Theorem \ref{thm: Edge coupling error estimate}.
    \begin{theorem}[Global error estimate]
        \label{thm: Edge coupling error estimate}
        Suppose for each edge $e$, there exists an edge function $\tilde{v}_e \in H_{00}^{1/2}(e)$ that satisfies \[\|P_e(\tilde{u}^{\sfh}-I_H\tilde{u}^{\sfh})-\tilde{v}_e\|_{\cH^{1/2}(e)} \leq \epsilon_{e}\, .\] 
      We identify $\tilde{v}_e$ with its zero extension  to $E_H$. Let $v_e \in H_0^1(\Omega)$ be the $a$-harmonic extension of $\tilde{v}_e$ to $\Omega$. Then, we have
        \[\|u^{\sfh}-I_Hu^{\sfh}-\sum_{e \in \cE_H} v_e\|^2_{H_a^1(\Omega)}\leq C_{\mathrm{mesh}}\sum_{e \in \cE_H} \epsilon_{e}^2\, ,\]
    \end{theorem}
    where $C_{\mathrm{mesh}}$ is a constant that depends on the mesh type only.
    \begin{proof}[Proof of Theorem \ref{thm: Edge coupling error estimate}]
    We decompose the energy norm into the contribution from each element $T \in \cT_H$:
    \[\|u^{\sfh}-I_Hu^{\sfh}-\sum_{e \in \cE_H} v_e\|^2_{H_a^1(\Omega)}=\sum_{T\in\cT_H}\|u^{\sfh}-I_Hu^{\sfh}-\sum_{e \sim T} v_e\|^2_{H_a^1(T)}\, , \]
    where we have used the fact that $v_e = 0$ in $T$ if $e$ and $T$ are not neighbors.
    
   Let us fix an element $T$. For each $e \sim T$, the trace of the function $u^{\sfh}-I_Hu^{\sfh}-\sum_{e \in T} v_e$ on $e$ is $\tilde{u}^{\sfh}-I_H\tilde{u}^{\sfh}-\tilde{v}_e\in H^{1/2}_{00}(e)$. We can extend this trace to $\partial T\backslash e$ by $0$ to get an $H^{1/2}(\partial T)$ boundary data. Then, this boundary data can be used to define an $a$-harmonic function in $T$. Using the triangle inequality and the Cauchy-Schwarz inequality after squaring both sides, we get
     \[ \|u^{\sfh}-I_Hu^{\sfh}-\sum_{e \sim T} v_e\|^2_{H_a^1(T)}\leq C_{\text{mesh}}\sum_{e \sim T} \|P_e(\tilde{u}^{\sfh}-I_H\tilde{u}^{\sfh})-\tilde{v}_e\|^2_{\cH_T^{1/2}(e)}\, ,\]
     where we have used the definition of the $\cH_T^{1/2}(e)$ norm. The constant $C_{\text{mesh}}$ depends on the mesh type only; for example $C_{\text{mesh}}=4$ for the quadrilateral mesh and $C_{\text{mesh}}=3$ for the triangular mesh.
    Then, we sum the above inequality over all $T \in \cT_H$, which yields
    \begin{equation} \label{eq1}
\begin{split}
 \|u^{\sfh}-I_Hu^{\sfh}-\sum_{e \in \cE_H} v_e\|^2_{H_a^1(\Omega)}&\leq C_{\mathrm{mesh}}\sum_{T\in\cT_H}\sum_{e \sim T} \|P_e(\tilde{u}^{\sfh}-I_H\tilde{u}^{\sfh})-\tilde{v}_e\|^2_{\cH_T^{1/2}(e)} \\
 & =C_{\mathrm{mesh}}\sum_{e \in \cE_H}\|P_e(\tilde{u}^{\sfh}-I_H\tilde{u}^{\sfh})-\tilde{v}_e\|^2_{\cH^{1/2}(e)} \\
 & \leq C_{\mathrm{mesh}}\sum_{e \in \cE_H} \epsilon_{e}^2\, .
\end{split}
\end{equation}
The proof is completed.
    \end{proof}
    Theorem \ref{thm: Edge coupling error estimate} implies that the global error can be localized onto edges. Thus, our target of the enrichment should be to ensure $\|P_e(\tilde{u}^{\sfh}-I_H\tilde{u}^{\sfh})-\tilde{v}_e\|_{\cH^{1/2}(e)} \leq \epsilon_{e}$ for some desired $\epsilon_e$. In the following, we formulate this condition in terms of edge basis functions, which leads to the solution accuracy of the Galerkin method.
     
    First, let us recall that edge basis functions have two parts. The first part is for interpolation; we denote all of them by $\tilde{V}_H^1=\operatorname{span}~\{\psi_i\}_i$, as in Subsection \ref{subsec: interpolation part}. The second part is for enrichment as in Subsection \ref{subsec: enrichment part}. For each edge $e$, the space of enrichment edge basis functions is denoted by $\tilde{V}^{2}_{H,e} \subset H_{00}^{1/2}(e)$; the union over all edges is $\tilde{V}^{2}_{H}=\bigcup_e \tilde{V}^{2}_{H,e} $. Then, the total edge approximation space is $\tilde{V}_H=\tilde{V}_H^1 \bigcup \tilde{V}_H^2 \subset H^{1/2}(E_H)$. The $a$-harmonic extensions of its functions into $\Omega$ constitutes the approximation space $V_H$. Based on Theorem \ref{thm: Edge coupling error estimate}, we get the following theorem:
    \begin{theorem}[Global error estimate of the Galerkin solution]
    \label{thm:Edge coupling error estimate: basis function version} 
    Suppose for each $e$, we have
    \[\min_{\tilde{v}_e \in \tilde{V}^2_{H,e}} \|P_e(u-I_Hu)-\tilde{v}_e\|_{\cH^{1/2}(e)}\leq \epsilon_e\, ,\]
    then using $V_H$ in the weak formulation \eqref{eqn: weak formulation} leads to a solution $u_H$ that satisfies
    \[\|u^\sfh - u_H\|_{H_a^1(\Omega)}^2\leq C_{\mathrm{mesh}}\sum_{e\in \cE_H} \epsilon_e^2\, . \]
    \end{theorem}
    \begin{proof}
    The proof is completed by observing that $P_e(u-I_Hu)=P_e(\tilde{u}^{\sfh}-I_H\tilde{u}^{\sfh})$ and using the property \eqref{eqn: refined Galerkin projection} together with Theorem \ref{thm: Edge coupling error estimate}.
    \end{proof}
    Theorem  \ref{thm:Edge coupling error estimate: basis function version} also implies the accuracy for the exact solution $u$; see the remark below.
    \begin{remark}
    \label{rmk: add bubble part}
If we add the bubble part $u^\sfb$, then we get the overall error estimate
\begin{equation}
    \|u-u^\sfb-u_H\|_{H_a^1(\Omega)}^2\leq C_{\mathrm{mesh}}\sum_{e \in \cE_H} \epsilon_e^2\, ,
\end{equation}
and if we do not compute the bubble part, we have the overall error
\begin{equation}
    \|u-u_H\|_{H_a^1(\Omega)}^2\leq C_{\mathrm{mesh}}\sum_{e \in \cE_H} \epsilon_e^2 +CH^2\|f\|_{L^2(\Omega)}^2\, ,
\end{equation}
due to the orthogonality and the estimate $\|u^\sfb\|_{H_a^1(\Omega)}\leq CH\|f\|_{L^2(\Omega)}$.
\end{remark}

In the next section, we will discuss how to choose the space $\tilde{V}^2_{H,e}$ so that it satisfies the above condition and in the meantime, can be efficiently computed using only local information of the equation.
    
    \section{Local Approximation via Oversampling} 
    \label{sec: Local Approximation via Oversampling}
  This section is aimed to study how to choose a local approximation space $\tilde{V}^2_{H,e}$ such that
  \begin{equation}
  \label{eqn: local error indicator}
      \min_{\tilde{v}_e \in \tilde{V}^2_{H,e}} \|P_e(u-I_Hu)-\tilde{v}_e\|_{\cH^{1/2}(e)}\leq \epsilon_e
  \end{equation}
   for some desired $\epsilon_e$. We call $\epsilon_e$ the local error indicator on $e$. For this problem, we can gain useful intuitions by assuming $a$ smooth first. In this case, $u \in H^2(\Omega)$ since $f \in L^2(\Omega)$. If we choose $\tilde{V}^2_{H,e}$ as the space of quadratic polynomials on $e$, then standard results in the Sobolev space will imply that \eqref{eqn: local error indicator} holds with $\epsilon_e=CH\|u\|_{H^2(\omega_e)}$, where $\omega_e$ is some domain in $\Omega$ that contains $e$, and $C$ is a constant independent of $H$ and $u$. Here, $\omega_e$ can be chosen as the closure of $\rN(e,\cT_H)$. For such a choice, the sum  $\sum_{e\in\cE_H} \epsilon_e^2$ yields $O(H^2\|u\|_{H^2(\Omega)}^2)=O(H^2\|f\|_{L^2(\Omega)}^2)$. Hence, finally we get $O(H)$ accuracy in the energy norm.
   
   For the more challenging case $a \in L^{\infty}(\Omega)$, we will also take a similar $\omega_e \supset e$.
   This $\omega_e$ is generally referred to as the \textit{oversampling} domain of the edge $e$; see the next subsection for details. The basic intuition is that for a coarse scale function such as the $a$-harmonic function, its behavior on $e$ can be controlled very well by that in the oversampling region $\omega_e$. Due to this reason, we could construct basis functions to approximate the coarse solution on $e$ in an exponentially efficient manner by using some information of the equation in $\omega_e$. 
   
   Historically, the idea of oversampling was proposed in \cite{hou_multiscale_1997} to reduce the resonance error in MsFEM.
   
    \subsection{Oversampling and Nearly Exponential Convergence}
    
    We consider an oversampling domain $ \omega_e$ for each $e \in \cE_H$. In principle, any $\omega_e$ that contains $e$ 
    can be used. 
    Here, for simplicity of analysis, we set 
    \begin{equation}
    \label{eqn: os domain 1 layer}
        \omega_e=\overline{\bigcup \{T\in \cT_H: \overline{T} \cap e \neq \emptyset\}}\, .
    \end{equation}
    This choice of $\omega_e$ makes $e$ lie in the interior of $\omega_e$ if $e \cap \partial \Omega =\emptyset$.
    We call such $e$ an \textit{interior edge}. If $e \cap \partial \Omega \neq \emptyset$ we call it an \textit{edge connected to the boundary}. An illustration of oversampling domains for a quadrilatenal mesh is in Figure \ref{fig:os domain}.
    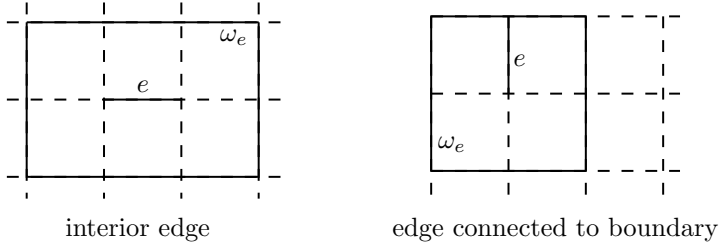
\begin{figure}[t]
        \centering

\tikzset{every picture/.style={line width=0.75pt}} 

\begin{tikzpicture}[x=0.75pt,y=0.75pt,yscale=-1,xscale=1]

\draw  [draw opacity=0][dash pattern={on 4.5pt off 4.5pt}] (60,64) -- (201.5,64) -- (201.5,163) -- (60,163) -- cycle ; \draw  [dash pattern={on 4.5pt off 4.5pt}] (70,64) -- (70,163)(109,64) -- (109,163)(148,64) -- (148,163)(187,64) -- (187,163) ; \draw  [dash pattern={on 4.5pt off 4.5pt}] (60,74) -- (201.5,74)(60,113) -- (201.5,113)(60,152) -- (201.5,152) ; \draw  [dash pattern={on 4.5pt off 4.5pt}]  ;
\draw    (109,113) -- (148,113) ;
\draw    (70,74) -- (187,74) ;
\draw    (187,74) -- (187,152) -- (70,152) -- (70,74) ;
\draw  [draw opacity=0][dash pattern={on 4.5pt off 4.5pt}] (274,71) -- (399.5,71) -- (399.5,162) -- (274,162) -- cycle ; \draw  [dash pattern={on 4.5pt off 4.5pt}] (274,71) -- (274,162)(313,71) -- (313,162)(352,71) -- (352,162)(391,71) -- (391,162) ; \draw  [dash pattern={on 4.5pt off 4.5pt}] (274,71) -- (399.5,71)(274,110) -- (399.5,110)(274,149) -- (399.5,149) ; \draw  [dash pattern={on 4.5pt off 4.5pt}]  ;
\draw    (313,71) -- (313,110) ;
\draw    (274,71) -- (274,149) -- (352,149) -- (352,71) -- cycle ;

\draw (124,102) node [anchor=north west][inner sep=0.75pt]   [align=left] {$\displaystyle e$};
\draw (166,76) node [anchor=north west][inner sep=0.75pt]   [align=left] {$\displaystyle \omega _{e}$};
\draw (88,171) node [anchor=north west][inner sep=0.75pt]   [align=left] {interior edge};
\draw (253,171) node [anchor=north west][inner sep=0.75pt]   [align=left] {edge connected to boundary};
\draw (314,89) node [anchor=north west][inner sep=0.75pt]   [align=left] {$\displaystyle e$};
\draw (276,130) node [anchor=north west][inner sep=0.75pt]   [align=left] {$\displaystyle \omega _{e}$};

\end{tikzpicture}
        \caption{Illustration of oversampling domains}
        \label{fig:os domain}
    \end{figure}
    
    Inspired by the discussions above, we will choose $\epsilon_e$ to depend on some norm of the function $u$ in $\omega_e$. Let us first write $P_e(u-I_Hu)$ in a form that depends explicitly on $\omega_e$. Recall that in Subsection \ref{subsec: Coarse-Fine Decomposition of Solution Space}, we use $u_T^{\sfh}, u_T^{\sfb}$ to denote the local $a$-harmonic part and bubble part of $u$ in $T$. Here, we will use $u_{\omega_e}^\sfh$ and $u_{\omega_e}^\sfb$ for the local $a$-harmonic part and bubble part of $u$ in $\omega_e$, i.e., equation \eqref{eqn: harmonic-bubble splitting} holds with $T$ replaced by $\omega_e$. Then, we have 
    \begin{equation}
    \label{eqn: oversampling, decomposition}
        P_e(u-I_Hu)=P_e(u_{\omega_e}^{\sfh}-I_Hu_{\omega_e}^{\sfh})+ P_e(u_{\omega_e}^{\sfb}-I_Hu_{\omega_e}^{\sfb})\, .
    \end{equation}
    In this way, we express the target function on $e$ using the information of $u$ in the oversampling domain $\omega_e$. Two terms emerge in \eqref{eqn: oversampling, decomposition}, and we shall approximate them separately.
    \subsubsection{Exponential Efficiency for a-Harmonic Functions}
    \label{subsec: Compactness of $a$-Harmonic Functions}
    In the first term of \eqref{eqn: oversampling, decomposition}, $u_{\omega_e}^{\sfh}$ is an $a$-harmonic function in $\omega_e$. Let $U(\omega_e) \subset H^1(\omega_e)$ be the space of all $a$-harmonic functions in $\omega_e$ such that (1) if $\omega_e \cap \partial \Omega \neq \emptyset$, then these functions vanish on $\omega_e \cap \partial \Omega$, and (2) if $\omega_e \cap \partial \Omega = \emptyset$, these functions are identified as equivalent modulus a constant. Clearly, $u_{\omega_e}^{\sfh}$ can be seen as an element in $U(\omega_e)$. When equipped with the $H_a^1(\omega_e)$ norm, we write this space as $(U(\omega_e), \|\cdot\|_{H_a^1(\omega_e)})$; this is a Hilbert space and $\|\cdot\|_{H_a^1(\omega_e)}$ is a norm due to the standard Poincar\'e inequality. For the $H_{00}^{1/2}(e)$ space equipped with the $\cH^{1/2}(e)$ norm, we write $(H_{00}^{1/2}(e),\|\cdot\|_{\cH^{1/2}(e)})$. Consider the operator \[R_e: (U(\omega_e),\|\cdot\|_{H_a^1(\omega_e)}) \to (H_{00}^{1/2}(e),\|\cdot\|_{\cH^{1/2}(e)})\, ,\]
    such that $R_e v=P_e(v-I_Hv)$ for any $v \in (U(\omega_e),\|\cdot\|_{H_a^1(\omega_e)})$. This operator is well defined because $R_e c=0$ for any constant $c \in \bR$. A very important property is that singular values of $R_e$ decay nearly exponentially fast; see the following Theorem \ref{thm: svd exponential decay of Re}. We recall the definition of the constants $c_0$ and $c_1$; they are dependent on the mesh property; see Subsection \ref{subsec:elements}.
     \begin{theorem}
     \label{thm: svd exponential decay of Re}
     For each $e \in \cE_H$, the operator $R_{e}$ is bounded and compact. Let the left singular vectors and singular values pair of $R_{e}$ be $\{\tilde{v}_{e,m},\lambda_{e,m}\}_{m \in \bN}$, in which $\tilde{v}_{e,m} \in H_{00}^{1/2}(e)$ and the sequence $\{\lambda_{e,m}\}_{m \in \bN}$ is in a descending order. Then, for any $\epsilon>0$, there exists an $N_{\epsilon}>0$, such that for all $m > N_{\epsilon}$, it holds that
     \begin{equation}
     \label{eqn: upper bound of singular value}
         \lambda_{e,m}\leq C\exp\left(-m^{(\frac{1}{d+1}-\epsilon)}\right)\, ,
     \end{equation}
     where $N_{\epsilon}, C$ depends on $d, a_{\min}, a_{\max}, c_0$ and $c_1$.
     
    Therefore, if we set $W_{e,m}=\mathrm{span}~\{\tilde{v}_{e,k}\}_{k=1}^{m-1}$ for some $m>N_{\epsilon}$, then we have
    \begin{equation}
    \label{eqn: svd approximation property}
        \min_{\tilde{v}_{e} \in W_{e,m}} \|R_e v-\tilde{v}_{e}\|_{\cH^{1/2}(e)}\leq C\exp\left(-m^{(\frac{1}{d+1}-\epsilon)}\right)\|v\|_{H_a^1(\omega_e)}\, ,
    \end{equation}
    for any $v \in (U(\omega_e),\|\cdot\|_{H_a^1(\omega_e)})$.
    \end{theorem}
    The proof is deferred to Subsection \ref{subsec: Proof of The Exponential Decay}. The approximation property \eqref{eqn: svd approximation property} can be seen as a consequence of the decay \eqref{eqn: upper bound of singular value}. We remark that approximation property like \eqref{eqn: svd approximation property} can also be phrased through the language of Kolmogorov's $n$-widths \cite{pinkus2012n, melenk2000n}, as is used in \cite{babuska2011optimal}.
    
    We discuss the implication of this theorem in the following. Taking $v=u_{\omega_e}^{\sfh}$ in \eqref{eqn: svd approximation property} leads to
    \begin{equation}
    \label{eqn: a harmonic part exponential efficiency}
          \min_{\tilde{v}_{e} \in W^m_{e}} \|P_e(u_{\omega_e}^{\sfh}-I_Hu_{\omega_e}^{\sfh})-\tilde{v}_{e}\|_{\cH^{1/2}(e)}\leq C\exp\left(-m^{(\frac{1}{d+1}-\epsilon)}\right)\|u_{\omega_e}^{\sfh}\|_{H_a^1(\omega_e)}\, .
    \end{equation}
    Thus, the space of singular vectors $W_{e,m}$ can approximate the first term in \eqref{eqn: oversampling, decomposition} very well; the approximation error decays nearly exponentially regarding $m$. Moreover, since $R_e$ is a local operator, basis functions in $W_{e,m}$ can be efficiently computed by solving a local singular value decomposition (SVD) problem.
    \begin{remark}
    \label{rmk:choice of oversampling}
    The scalar $N_{\epsilon}$ in Theorem \ref{thm: svd exponential decay of Re} will indeed depend on the relative lengthscale of $e$ and $\omega_e$. Here, because we choose the oversampling domain in a specific form \eqref{eqn: os domain 1 layer} and the mesh is uniform and shape regular, the relative lengthscale can be treated as a constant independent of $H$. In general, if we increase the lengthscale of $\omega_e$ to make it larger, the decay of singular values of $R_e$ will become faster, leading to a smaller $N_{\epsilon}$.
    \end{remark}
    \subsubsection{The Oversampling Bubble Part}
    \label{subsec: The Oversampling Bubble Part}
    In the second term of \eqref{eqn: oversampling, decomposition}, we call $u_{\omega_e}^{\sfb}$ the oversampling bubble part. By definition, this part can be efficiently computed by solving local elliptic equations in $\omega_e$ with right-hand side $f$. Moreover, it is small in the sense that we are able to provide a priori bound of order $H$; see Proposition \ref{prop: os bubble small}.
    \begin{proposition}
For each $e \in \cE_H$, the following estimate holds for the oversampling bubble part:
    \label{prop: os bubble small}
    \[\|P_e(u_{\omega_e}^{\sfb}-I_Hu_{\omega_e}^{\sfb})\|_{\cH^{1/2}(e)}\leq CH\|f\|_{L^2(\omega_e)}\, ,\]
    where $C$ is a constant independent of $u$ and $H$.
    \end{proposition}
    The proof of Proposition \ref{prop: os bubble small} is deferred to Subsection \ref{subsec: Proof of The Exponential Decay}.
    \subsubsection{Exponentially Efficient Local Approximation Spaces} Based on Subsections \ref{subsec: Compactness of $a$-Harmonic Functions} and \ref{subsec: The Oversampling Bubble Part}, we will select the local approximation space as
    \[\tilde{V}_{H,e,m}^2:=W_{e,m} \bigcup \{P_e(u_{\omega_e}^{\sfb}-I_Hu_{\omega_e}^{\sfb})\}\, .\]
    This space can be locally computed by SVD and solving an elliptic equation. Due to \eqref{eqn: oversampling, decomposition} and \eqref{eqn: oversampling, decomposition}, we have the error estimate:
    \[\min_{\tilde{v}_e \in \tilde{V}^2_{H,e,m}} \|P_e(u-I_Hu)-\tilde{v}_e\|_{\cH^{1/2}(e)}\leq C\exp\left(-m^{(\frac{1}{d+1}-\epsilon)}\right)\|u_{\omega_e}^{\sfh}\|_{H_a^1(\omega_e)}\, .\]
    Hence, the local error indicator is $\epsilon_e=C\exp\left(-m^{(\frac{1}{d+1}-\epsilon)}\right)\|u_{\omega_e}^{\sfh}\|_{H_a^1(\omega_e)}$. By Theorem \ref{thm:Edge coupling error estimate: basis function version}, summing these errors leads to the final error estimate of the coarse part $u^\sfh$; see Proposition \ref{prop: exponential error estimate}. 

    \begin{proposition}
    \label{prop: exponential error estimate}
    Let the space of the enrichment part $\tilde{V}^{2}_{H,m}=\bigcup_e \tilde{V}^{2}_{H,e,m} $. The total edge approximation space is $\tilde{V}_{H,m}=\tilde{V}_H^1 \bigcup \tilde{V}_{H,m}^2 \subset H^{1/2}(E_H)$. Let $V_{H,m}$ constitute of $a$-harmonic extensions of functions in $\tilde{V}_{H,m}$ into $\Omega$. Then, using $V_{H,m}$ in the weak formulation \eqref{eqn: weak formulation} leads to a solution $u_{H,m}$ that satisfies
    \[\|u^\sfh - u_{H,m}\|_{H_a^1(\Omega)} \leq C'\exp\left(-m^{(\frac{1}{d+1}-\epsilon)}\right)\|f\|_{L^2(\Omega)}\, , \]
    where $C'$ is a constant independent of $u, m$ and $H$.
    \end{proposition}
    \begin{proof}
    Based on Theorem \ref{thm:Edge coupling error estimate: basis function version}, we have
    \[ \|u^\sfh - u_{H,m}\|_{H_a^1(\Omega)}^2 \leq C_{\mathrm{mesh}}\sum_{e \in \cE_H} \epsilon_e^2 \, . \]
    We can bound the sum of the local errors:
    \begin{align*}
    \sum_{e \in \cE_H} \epsilon_e^2&=C^2\sum_{e \in \cE_H} \exp\left(-2m^{(\frac{1}{d+1}-\epsilon)}\right)\|u_{\omega_e}^{\sfh}\|_{H_a^1(\omega_e)}^2\\     
    &\leq C^2\sum_{e \in \cE_H} \exp\left(-2m^{(\frac{1}{d+1}-\epsilon)}\right)\|u\|_{H_a^1(\omega_e)}^2\\
    &\leq C^2C_1\exp\left(-2m^{(\frac{1}{d+1}-\epsilon)}\right)\|u\|_{H_a^1(\Omega)}^2\, ,
    \end{align*}
    where we have used the fact that though different oversampling domains $\omega_e$ can have overlapping, each element shall be only counted by a finite number of times bounded by a constant independent of the mesh size $H$, for our choice of domain in \eqref{eqn: os domain 1 layer}. 
    
    Finally, the proof is completed by using the elliptic estimate \[\|u\|^2_{H_a^1(\Omega)}\leq C_2\|f\|_{L^2(\Omega)}^2\] for some $C_2$ independent of $u$ and $H$, and the constant $C'=C\sqrt{C_{\mathrm{mesh}}C_1C_2}$.
    \end{proof}
    
    Note that $\operatorname{dim}(\tilde{V}^{2}_{H,e,m})=m$, so in Proposition \ref{prop: exponential error estimate} we have achieved nearly exponential convergence of the approximation error with respect to computational degrees of freedom. We provide several remarks below. We use $C$ to represent a generic constant that is independent of $u, m$ and $H$.
    \begin{remark}
    Different choices of the oversampling domains shall lead to different constant $C_1$ in the proof of Proposition \ref{prop: exponential error estimate}. In general, if we increase the lengthscale of $\omega_e$ to make the oversampling domain larger, the constant $C_1$ shall become larger, while in Remark \ref{rmk:choice of oversampling} we will get a smaller $N_{\epsilon}$. So in general there would be a trade-off between the decay of singular values and the overlapping information in our choice of the oversampling domains.
    \end{remark}
    \begin{remark}
    \label{rmk: with or without bubble}
    As in Remark \ref{rmk: add bubble part}, if we add the bubble part $u^\sfb$, then the overall accuracy will be
    \begin{equation}
    \label{eqn: error estimate adaptive}
    \|u - u_{H,m}-u^\sfb\|_{H_a^1(\Omega)} \leq C\exp\left(-m^{(\frac{1}{d+1}-\epsilon)}\right)\|f\|_{L^2(\Omega)}\, ;    
    \end{equation}
    that is, we get nearly exponential accuracy for approximating the solution $u$. If we do not add the bubble part, we obtain
    \[\|u - u_{H,m}\|_{H_a^1(\Omega)} \leq (C\exp\left(-m^{(\frac{1}{d+1}-\epsilon)}\right)+CH)\|f\|_{L^2(\Omega)}\, . \]
    \end{remark}
    \begin{remark}
    The error estimate in \eqref{eqn: error estimate adaptive} explains the adaptivity of our edge basis functions. For general $f\in L^2(\Omega)$, the exponential decay of the approximation error cannot be achieved if we do not adapt the method to the right hand side $f$. In connection to the adaptive finite element method, the error indicator function in our method is $\epsilon_e$ for each edge $e$, which will decrease as we increase the number of edge basis functions. We remark that it is also possible to truncate the singular values of $R_e$ to some threshold adaptively so that we can have a total control of the local error  $\epsilon_e$; see Remark \ref{rmk: spatial adaptivity}.
    \end{remark}
    \begin{remark}
    \label{rmk: no bubble and os bubble}
    If we do not include $P_e(u_{\omega_e}^{\sfb}-I_Hu_{\omega_e}^{\sfb})$ in our enrichment of edge basis functions, i.e., we use $W_{e,m}$ directly for the enrichment part, then we will get
     \[\min_{\tilde{v}_e \in W_{e,m}} \|P_e(u-I_Hu)-\tilde{v}_e\|_{\cH^{1/2}(e)}\leq C\exp\left(-m^{(\frac{1}{d+1}-\epsilon)}\right)\|u_{\omega_e}^{\sfh}\|_{H_a^1(\omega_e)}+CH\|f\|_{L^2(\omega_e)}\, .\]
     Thus, finally our computation of the coarse part $u^\sfh$ and the exact $u$ both will be subject to an error upper bounded by $(C\exp\left(-m^{(\frac{1}{d+1}-\epsilon)}\right)+CH)\|f\|_{L^2(\Omega)}$. 
     
     Our bound implies that $m=O(\log^{d+1}(1/H))$ suffices for $O(H)$ accuracy of the solution $u$ in the energy norm. To get this level of accuracy, we do not need the information of the right-hand side in constructing $V_H$, and we do not need to solve the fine scale bubble part $u^\sfb$.
     
     We will perform numerical experiments in the next section to demonstrate the importance of using bubble parts 
    for achieving the nearly exponential accuracy.
    \end{remark}
    \begin{remark}
    \label{rmk: spatial adaptivity}
    We can also add some spatial adaptivity to the implementation of this method. In \eqref{eqn: svd approximation property}, we are essentially using the approximation property of the singular vectors of $R_e$:
    \begin{equation}
        \min_{\tilde{v}_{e} \in W_{e,m}} \|R_e v-\tilde{v}_{e}\|_{\cH^{1/2}(e)}\leq \lambda_{e,m} \|v\|_{H_a^1(\omega_e)}\, .
    \end{equation}
    The local error indicator is $\epsilon_e=\lambda_{e,m} \|v\|_{H_a^1(\omega_e)}$. Now, we can choose not to use the analytic upper bound of $\lambda_{e,m}$ in \eqref{eqn: upper bound of singular value} directly; instead, we do SVD of $R_e$ and truncate the spectrum to a desired $m_e$ such that $\lambda_{e,m_e}$ is below some threshold of accuracy. In this way, we can control the local approximation errors on different edges adaptively. This approach has been adopted in \cite{hou2015optimal}. Naturally, it leads to a different number of basis functions $m_e$ for each edge $e$. The final accuracy will be of order $(\max_{e\in \cE_H} \lambda_{e,m_e})\|f\|_{L^2(\Omega)}$ according to the proof of Proposition \ref{prop: exponential error estimate}. Thus, we are able to achieve an overall accuracy that is  adaptive to $\lambda_{e,m_e}$ in the local error indicator. This adaptivity can be very helpful from a practical point of view: when $a(x)$ is smooth in some local region, we expect $\lambda_{e,m_e}$ to be smaller than the corresponding $\lambda_{e,m_e}$ in the region where $a(x)$ is rough, so the number of basis functions in the region where $a(x)$ is smooth can be reduced. Therefore, our method not only provably handles the rough coefficients, but can also adapt to the local smoothness of the coefficients.
    This explains the adaptive nature of our edge basis functions.
    \end{remark}
    \subsection{Connection to Existing Works}
    \label{sec: Connection to Existing Works}
    In this subsection, we discuss the connection of our method to existing works.
    The compactness property of $a$-harmonic functions in Theorem \ref{thm: svd exponential decay of Re} is motivated by the work \cite{babuska2011optimal}, where the compactness of a restriction operator of $a$-harmonic functions on concentric regions is studied and nearly exponential decay of singular values is established.
    
    According to the discussion at the beginning of this section, we can interpret our method as choosing $\epsilon_e$ to depend on the norm $\|u_{\omega_e}^{\sfh}\|_{H_a^1(\omega_e)}$. More precisely, we set $\epsilon_e=\delta_e \|u_{\omega_e}^{\sfh}\|_{H_a^1(\omega_e)}$ for some small $\delta_e$ and seek for a space $\tilde{V}^2_{H,e}$ such that 
     \[\min_{\tilde{v}_e \in \tilde{V}^2_{H,e}} \|P_e(u-I_Hu)-\tilde{v}_e\|_{\cH^{1/2}(e)}\leq \delta_e \|u_{\omega_e}^{\sfh}\|_{H_a^1(\omega_e)}\, .\]
     In such a setting, the optimal way of getting a small $\delta_e$ is to do SVD for the operator $R_e$ with domain $(U(\omega_e),\|\cdot\|_{H_a^1(\omega_e))}$. The oversampling bubble part can also be readily identified by using the decomposition \eqref{eqn: oversampling, decomposition}.
     
     It is possible to use other norms in $\omega_e$. For example, in the work \cite{hou2015optimal}, the norm being used is
    \begin{equation*}
        \|v\|_{\omega_e}^2:= \int_{\omega_e} \left( a |\nabla v_{\omega_e}^{\sfh}|^2+ (v_{\omega_e}^{\sfh})^2 + [\nabla \cdot (a \nabla v)]^2\right)\, ,
    \end{equation*}
    for any $v \in H^1(\omega_e)$. For such a choice, they target at finding a space $\tilde{V}^2_{H,e}$ so that
     \[\min_{\tilde{v}_e \in \tilde{V}^2_{H,e}} \|P_e(u-I_Hu)-\tilde{v}_e\|_{\cH^{1/2}(e)}\leq \delta_e \|u\|_{\omega_e}\, ,\]
    for some small $\delta_e$. In this case, a different operator involving the norm $\|\cdot\|_{\omega_e}$ will be defined, and the corresponding SVD needs to be performed. 
    
    Nevertheless, we believe that the choice $\|u_{\omega_e}^{\sfh}\|_{H_a^1(\omega_e)}$ in this paper is the most natural one that can lead to nearly exponential convergence. It makes the fact explicit that adding an oversampling bubble term $P_e(u_{\omega_e}^{\sfb}-I_Hu_{\omega_e}^{\sfb})$ can guarantee the exponential accuracy in theory. This fact is not apparent using the norm $\|\cdot\|_{\omega_e}$. Indeed, the method in \cite{hou2015optimal} does not lead to a perfect nearly exponential accuracy because the information of $f$ is not incorporated into the construction of edge basis functions.
    \subsection{Proof of the Nearly Exponential Accuracy} 
    \label{subsec: Proof of The Exponential Decay}
    In this subsection, we prove the main results of this section, Theorem \ref{thm: svd exponential decay of Re} and Proposition \ref{prop: os bubble small}. They both contain statements for every edge $e \in \cE_H$. For interior edges and edges connected to the boundary, the treatments will be slightly different. We present the results for interior edges first, and then for edges connected to the boundary.
    
    In both cases, we will begin with two useful lemmas. The first lemma gives an upper bound for $\|P_e (v-I_H v)\|_{\cH^{1/2}(e)}$ by some norm of $v$ in a larger domain that contains $e$. The second lemma demonstrates the decay of singular values of a specific restriction operator acting on $a$-harmonic functions; this lemma appeared first in Theorem 3.3 of \cite{babuska2011optimal}. 
    
    \subsubsection{Interior Edges}
    For an interior edge $e$, the oversampling region $\omega_e$ contains $e$ strictly in the interior. This fact yields the major difference between proofs for interior edges and edges connected to the boundary. The latter case follows a similar argument, and the corresponding proofs will be presented in Subsection \ref{subsec: Proof of The Exponential Decay1}.
    
    First of all, we discuss some geometric relation between $e$ and $\omega_e$ that will be needed in our analysis; Figure \ref{fig:omega omega star} illustrates our ideas for a uniform quadrilaternel mesh, and for more general shape regular mesh the same construction holds. For each interior edge $e$, there exists two concentric rectangles $\omega \subset \omega^*$ with center being the midpoint $m_e$ of $e$, such that $e \subset \omega \subset \omega^* \subset \omega_e$. Namely the center $m_e$ is the center of gravity of $\omega$ and $\omega^*$. We require $\omega^* \cap \partial \Omega = \emptyset$. 
    Moreover, one side of $\omega$ and $\omega^*$ should be parallel to $e$. We introduce three parameters $l_1,l_2,l_3$ to specify and describe the geometry:
    \begin{figure}[t]
        \centering
\tikzset{every picture/.style={line width=0.75pt}} 

\begin{tikzpicture}[x=0.75pt,y=0.75pt,yscale=-1,xscale=1]

\draw  [draw opacity=0][dash pattern={on 4.5pt off 4.5pt}] (60,64) -- (201.5,64) -- (201.5,163) -- (60,163) -- cycle ; \draw  [dash pattern={on 4.5pt off 4.5pt}] (70,64) -- (70,163)(109,64) -- (109,163)(148,64) -- (148,163)(187,64) -- (187,163) ; \draw  [dash pattern={on 4.5pt off 4.5pt}] (60,74) -- (201.5,74)(60,113) -- (201.5,113)(60,152) -- (201.5,152) ; \draw  [dash pattern={on 4.5pt off 4.5pt}]  ;
\draw    (109,113) -- (148,113) ;
\draw    (70,74) -- (187,74) ;
\draw    (187,74) -- (187,152) -- (70,152) -- (70,74) ;
\draw  [draw opacity=0][dash pattern={on 4.5pt off 4.5pt}] (282,68) -- (407.5,68) -- (407.5,159) -- (282,159) -- cycle ; \draw  [dash pattern={on 4.5pt off 4.5pt}] (282,68) -- (282,159)(321,68) -- (321,159)(360,68) -- (360,159)(399,68) -- (399,159) ; \draw  [dash pattern={on 4.5pt off 4.5pt}] (282,68) -- (407.5,68)(282,107) -- (407.5,107)(282,146) -- (407.5,146) ; \draw  [dash pattern={on 4.5pt off 4.5pt}]  ;
\draw    (321,68) -- (321,107) ;
\draw    (282,68) -- (282,146) -- (360,146) -- (360,68) -- cycle ;
\draw  [dash pattern={on 0.84pt off 2.51pt}] (94.5,95) -- (159.5,95) -- (159.5,132) -- (94.5,132) -- cycle ;
\draw  [dash pattern={on 0.84pt off 2.51pt}] (79.75,84.5) -- (174.25,84.5) -- (174.25,142.5) -- (79.75,142.5) -- cycle ;
\draw  [dash pattern={on 0.84pt off 2.51pt}] (304.5,68) -- (336.5,68) -- (336.5,126) -- (304.5,126) -- cycle ;
\draw  [dash pattern={on 0.84pt off 2.51pt}] (292.5,67.5) -- (347.5,67.5) -- (347.5,139) -- (292.5,139) -- cycle ;

\draw (124,103) node [anchor=north west][inner sep=0.75pt]   [align=left] {$\displaystyle e$};
\draw (160,59) node [anchor=north west][inner sep=0.75pt]   [align=left] {$\displaystyle \omega _{e}$};
\draw (87,172) node [anchor=north west][inner sep=0.75pt]   [align=left] {interior edge};
\draw (257,174) node [anchor=north west][inner sep=0.75pt]   [align=left] {edge connected to boundary};
\draw (323,79) node [anchor=north west][inner sep=0.75pt]   [align=left] {$\displaystyle e$};
\draw (262,122) node [anchor=north west][inner sep=0.75pt]   [align=left] {$\displaystyle \omega _{e}$};
\draw (96,95) node [anchor=north west][inner sep=0.75pt]   [align=left] {$\displaystyle \omega $};
\draw (78.75,84.5) node [anchor=north west][inner sep=0.75pt]   [align=left] {$\displaystyle \omega ^{*}$ };
\draw (288,67) node [anchor=north west][inner sep=0.75pt]   [align=left] {$\displaystyle \omega ^{*}$ };
\draw (301,85) node [anchor=north west][inner sep=0.75pt]   [align=left] {$\displaystyle \omega $};

\end{tikzpicture}
        \caption{Geometric relation $e\subset\omega\subset \omega^*\subset \omega_e$}
        \label{fig:omega omega star}
    \end{figure}
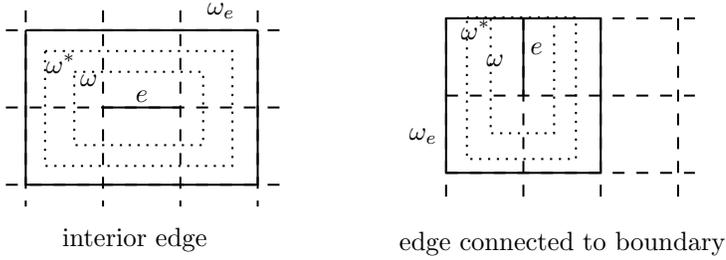
    \begin{enumerate}
        \item With respect to the center $m_e$, the two rectangles $\omega$ and $\omega^*$ are scaling equivalent, such that there exists $l_1>1$, $\omega^*-m_e=l_1\cdot (\omega-m_e)$. For our choice of $\omega_e$, the parameter $l_1$ can be selected to only depend on $c_0$ and $c_1$ in Subsection \ref{subsec:elements}. Here we use the notation that $t\cdot X :=\{tx:x\in X\}$ for a set $X$ and a scalar $t$.
        \item The ratio of $\omega$'s larger side length over the smaller side length is bounded by a uniform constant $l_2>1$ that depends on $c_0$ and $c_1$ only.
        \item There is a constant $l_3 > 1$ depending on $c_0$ and $c_1$ only such that $l_3\cdot e \subset \omega$. 
    \end{enumerate}
      We note that $l_1,l_2,l_3$ are universal constants for all interior edges. All three parameters depend on $c_0,c_1$ only.
    
Given the geometric relation, our first lemma is to bound $\|P_e (v-I_H v)\|_{\cH^{1/2}(e)}$ by some norm of $v$ in $\omega$, for $v \in H^1(\omega)$ and $\nabla \cdot (a\nabla v) \in L^2(\omega)$; see Lemma \ref{conjecture1} below.
    
    \begin{lemma}
\label{conjecture1}
There exists a constant $C$ dependent on $c_0, c_1, a_{\min}, a_{\max}$ such that the following estimate holds:
\begin{equation}
\label{localcon}
\left\|P_e (v-I_H v)\right\|_{\cH^{1/2}(e)}\leq C 
\left(\|v\|_{H_a^1(\omega)} + H\|\nabla \cdot (a\nabla v)\|_{L^2(\omega)}\right) \, ,
\end{equation}
for all $v \in H^1(\omega)$ and $\nabla \cdot (a\nabla v) \in L^2(\omega)$.
\end{lemma}
\begin{proof}
By construction, $l_3\cdot e\subset \omega$. We rescale the domain such that $e=[-1,1]$. Then $l_3\cdot e=[-l_3,l_3] \subset \omega$ and $m_e=0$ is the center of gravity of the rescaled $\omega$. Since the ratio of $\omega$'s larger side length over the smaller side length is bounded by $l_2$, we get \[\omega':=[-l_3,l_3]\times [-\frac{l_3}{l_2},\frac{l_3}{l_2}]\subset \omega\, .\]
For the rescaled domain, it suffices to prove
\begin{equation}
\label{eqn: target}
\left\|P_e (v-I_H v)\right\|_{\cH^{1/2}(e)}\leq C 
\left(\|v\|_{H_a^1(\omega')} + \|\nabla \cdot (a\nabla v)\|_{L^2(\omega')}\right) \, ,
\end{equation}
for a constant $C$ dependent on $c_0, c_1, a_{\min}, a_{\max}$, where we have utilized the different scaling property of different norms. We will omit $P_e$ in the following when there is no confusion. To prove \eqref{eqn: target}, first, we use 
Theorem 8.22 in \cite{gilbarg2015elliptic} to get a H\"older estimate of $v$:
\begin{equation}
\label{Holderalpha}
\|v\|_{C^\alpha(e)}\leq C\left(\|v\|_{L^2(\omega')}+\|\nabla \cdot (a\nabla v)\|_{L^2(\omega')} \right) \, ,
\end{equation}
   where $\|v\|_{C^\alpha(e)}=\|v\|_{L^{ \infty}(e)}+|v|_{C^\alpha(e)}$ such that 
   $|v|_{C^\alpha(e)}$ is the semi-$C^{\alpha}$ norm:
   \[ |v|_{C^\alpha(e)}:=\sup_{x,y\in e}\frac{|v(x)-v(y)|}{|x-y|^{\alpha}}\, .\]  The parameters $0<\alpha<1$ and $C$ are only dependent on the contrast of the coefficients $a$ and $c_0,c_1$. Our scaling argument is valid since a scaling transformation will not change the contrast of the coefficient $a$. 
   
   Because $I_H v$ is a linear function, it is in $C^1(e)$. We can bound $\|v-I_Hv\|_{L^{\infty}(e)}\leq 2\|v\|_{L^{\infty}(e)}$ and 
   \begin{equation}
   \label{eqn: C alpha w}
    |v-I_Hv|_{C^{\alpha}(e)}\leq |v|_{C^{\alpha}(e)}+|I_Hv|_{C^{\alpha}(e)}\leq |v|_{C^{\alpha}(e)}+2^{2-\alpha}\|v\|_{L^{\infty}(e)}   
   \end{equation}
    where we have used the fact that \[|I_Hv|_{C^{\alpha}(e)}\leq |v(1)-v(-1)|\cdot \sup_{x,y \in e} |x-y|^{1-\alpha} \leq 2^{2-\alpha}\|v\|_{L^{\infty}(e)}\, .\]
   We will use the H\"older estimate of $v-I_Hv$ on $e$ to prove \eqref{eqn: target}. Because the $\cH^{1/2}(e)$ norm and $H_{00}^{1/2}(e)$ are equivalent, we work on the $H_{00}^{1/2}(e)$ norm. Let $w=v-I_Hv$. According to the proof of Proposition \ref{prop: C alpha interpolation residue}, we have
   \begin{equation}
   \label{H00 1/2 norm}
        \|w\|_{H_{00}^{1/2}(e)}^2=\int_e|w(x)|^2\, \rd x+\int_e\int_e\frac{|w(x)-w(y)|^2}{|x-y|^2}\, \rd x\rd y+\int_{e} \frac{|w(x)|^2}{\dist(x,\partial e)}\, \rd x\, .
    \end{equation}
    The first term on the right-hand side of \eqref{H00 1/2 norm} is bounded by the third term, so we only need to upper bound the second and third terms. For the second term, we have 
    \begin{align*}
    \int_e\int_e\frac{|w(x)-w(y)|^2}{|x-y|^2}\, \rd x\rd y &=\int_{-1}^1\int_{-1}^1\frac{|w(x)-w(y)|^2}{|x-y|^2}\, \rd x\rd y\\
    &\leq 4|w|_{C^{\alpha}(e)}^2\\
    &\leq C\left(|v|_{C^{\alpha}(e)}^2+\|v\|_{L^{\infty}(e)}^2\right)\\
    &\leq C\left(\|v\|_{L^2(\omega')}^2+\|\nabla \cdot (a\nabla v)\|_{L^2(\omega')}^2 \right)\, ,
    \end{align*}
    where $C$ is a constant dependent on $\alpha, a_{\min}, a_{\max}$ and we have used \eqref{Holderalpha} and \eqref{eqn: C alpha w} in the last two inequalities. The constant $C$ can vary from line to line. For the third term in \eqref{H00 1/2 norm}, we have
    \begin{align*}
    \int_{e} \frac{|w(x)|^2}{\dist(x,\partial e)}\, \rd x &=\int_{-1}^0 \frac{|w(x)-w(-1)|^2}{|x+1|}\, \rd x+ \int_{0}^1 \frac{|w(x)-w(1)|^2}{|x-1|}\, \rd x\\
    &\leq |w|_{C^{\alpha}(e)}^2\left(\int_{-1}^0 |x+1|^{2\alpha-1}\,\rd x+\int_0^1 |x-1|^{2\alpha-1}\, \rd x \right)\\
    &\leq C|w|_{C^{\alpha}(e)}^2\\
    &\leq C\left(\|v\|_{L^2(\omega')}^2+\|\nabla \cdot (a\nabla v)\|_{L^2(\omega')}^2 \right)\, ,
    \end{align*}
    for some $C$ dependent on $\alpha, a_{\min}, a_{\max}$. Combining the estimates for these two terms, we arrive at
    \[\|w\|_{H_{00}^{1/2}(e)}\leq C\left(\|v\|_{L^2(\omega')}^2+\|\nabla \cdot (a\nabla v)\|_{L^2(\omega')}^2 \right)\, , \]
    for some $C$ dependent on $c_0, c_1, a_{\min}, a_{\max}$. The proof is completed.
\end{proof}

The second lemma is about the decay of singular values of a restriction operator acting on $a$-harmonic functions. Following a similar notation
in Subsection \ref{subsec: Compactness of $a$-Harmonic Functions}, we use $U(\omega^*)$ and $U(\omega)$ for the space of $a$-harmonic functions in $\omega^*$ and $\omega$ modulus a constant, respectively. We write $(U(\omega^*), \|\cdot\|_{H_a^1(\omega^*)})$ and $(U(\omega), \|\cdot\|_{H_a^1(\omega)})$ to be the corresponding spaces modulus a constant and equipped with the energy norm. The restriction operator is defined as 
\begin{equation}
\label{eqn: oversampling restriction operator}
    P_\omega: (U(\omega^*), \|\cdot\|_{H_a^1(\omega^*)}) \to (U(\omega), \|\cdot\|_{H_a^1(\omega)})\, ,
\end{equation}
such that $P_\omega u(x)=u(x)$ for $u \in (U(\omega^*), \|\cdot\|_{H_a^1(\omega^*)})$ and $x \in \omega$. Though not written explicitly, we should have in mind that $\omega$ and $\omega^*$ are associated with the interior edge $e$.
    
    \begin{lemma}
        \label{thm: decay of restriction operator for harmonic functions}
       For each interior edge $e$, the operator $P_{\omega}$ is bounded and compact.  Let the pairs of left singular vectors and singular values of $P_{\omega}$ be $\{v_{e,m},\mu_{e,m} \}_{m \in \bN}$ and the sequence $\{\mu_{e,m}\}_{m \in \bN}$ is in a descending order.  Then for any $\epsilon>0$, there exists an $N_{\epsilon}>0$, such that for all $m > N_{\epsilon}$, it holds that
     \begin{equation}
     \label{eqn: 1upper bound of singular value}
         \mu_{e,m}\leq \exp\left(-m^{(\frac{1}{d+1}-\epsilon)}\right)\, ,
     \end{equation}
     where $N_{\epsilon}$ depends on $d, a_{\min}, a_{\max}, c_0$ and $c_1$.
    
     Therefore, if we set $\Phi_{e,m}=\mathrm{span}~\{v_{e,k}\}_{k=1}^{m-1}$ for some $m>N_{\epsilon}$, then we have
    \begin{equation}
      \min_{\chi \in \Phi_{e,m}}\|P_\omega v-\chi\|_{H_a^1(\omega)}\leq \exp\left(-m^{(\frac{1}{d+1}-\epsilon)}\right) \|v\|_{H_a^1(\omega^*)}\, ,
    \end{equation}
    for any $v \in (U(\omega^*), \|\cdot\|_{H_a^1(\omega^*)})$.    
    \end{lemma}
    Lemma \ref{thm: decay of restriction operator for harmonic functions} is a generalization of Theorem 3.3 in \cite{babuska2011optimal} into cases where $\omega$ is no longer a cube but a shape regular rectangle. The result remains valid and the proof remains the same, where essentially we use the renowned Caccioppoli inequality and iteration arguments to obtain the quantification of such a compact embedding.

     Now, with Lemma \ref{conjecture1} and Lemma \ref{thm: decay of restriction operator for harmonic functions} in mind, we will prove Theorem \ref{thm: svd exponential decay of Re} and Proposition \ref{prop: os bubble small} below.
      
\begin{proof}[Proof of Theorem  \ref{thm: svd exponential decay of Re}]
First, we use Lemma \ref{thm: decay of restriction operator for harmonic functions} to get an $m-1$ dimensional space $\Phi_{e,m} \subset U(\omega)$, such that for any $v \in U(\omega^*)$, we have
\begin{equation}
\label{kolomogrov}
      \min_{\chi \in \Phi_{e,m}}\|P_\omega v-\chi\|_{H_a^1(\omega)}\leq \exp\left(-m^{(\frac{1}{d+1}-\epsilon)}\right) \|v\|_{H_a^1(\omega^*)}\, .
    \end{equation}
Combined with the fact that $\omega^*\subset \omega_e$, we get that for any $v \in U(\omega_e)$, it holds 
\begin{equation}
\label{kolomogrov1}
      \min_{\chi \in \Phi_{e,m}}\|P_\omega v-\chi\|_{H_a^1(\omega)}\leq \exp\left(-m^{(\frac{1}{d+1}-\epsilon)}\right) \|v\|_{H_a^1(\omega_e)}\, .
    \end{equation}
Setting $v=P_\omega v-\chi$ in Lemma \ref{conjecture1} leads to
\begin{align}
\label{eqn: e to omega}
    \|P_e\left((P_\omega v-\chi)-I_H(P_\omega v-\chi)\right)\|_{\cH^{1/2}(e)}\leq C\|P_\omega v-\chi\|_{H_a^1(\omega)}\, ,
\end{align}
 where we have used the fact that $P_\omega v-\chi$ is $a$-harmonic in $\omega$. We also have the relation
 \begin{equation}
 \label{eqn: equality}
    P_e\left((P_\omega v-\chi)-I_H(P_\omega v-\chi)\right)=P_e(v-I_Hv)-P_e(\chi-I_H\chi)=R_ev-P_e(\chi-I_H\chi)\, .
 \end{equation}
 Let $W_{e,m} = P_e(\Phi_{e,m}-I_H\Phi_{e,m}) \subset H_{00}^{1/2}(e)$. Combining \eqref{kolomogrov1}, \eqref{eqn: e to omega} and \eqref{eqn: equality}, we get 
 \begin{equation}
     \min_{w\in W_{e,m}} \|R_e v-w\|_{\cH^{1/2}(e)}\leq C\exp\left(-m^{(\frac{1}{d+1}-\epsilon)}\right) \|v\|_{H_a^1(\omega_e)}\, ,
 \end{equation}
for all $m > N_{\epsilon}$, where $C$ depends on $c_0, c_1, a_{\min}, a_{\max}$. 
This implies the upper bound of the singular values of $R_e$:
 \begin{equation}
         \lambda^{m}_e\leq C\exp\left(-m^{(\frac{1}{d+1}-\epsilon)}\right)\, .
     \end{equation}
The proof is completed.
\end{proof}
Then, we present the proof for Proposition \ref{prop: os bubble small}.
\begin{proof}[Proof of Proposition \ref{prop: os bubble small}]
Taking $v=u^\sfb_{\omega_e}$ in Lemma \ref{conjecture1}, we get
\begin{equation}
\left\|P_e (u^\sfb_{\omega_e}-I_H u^\sfb_{\omega_e})\right\|_{\cH^{1/2}(e)}\leq C 
\left(\|u^\sfb_{\omega_e}\|_{H_a^1(\omega)} + H\|f\|_{L^2(\omega)}\right) \, .
\end{equation}
The proof is completed by using the elliptic estimate
    $\|u^\sfb_{\omega_e}\|_{H_a^1(\omega_e)}\leq CH\|f\|_{L^2(\omega_e)}$ and the fact $\omega \subset \omega_e$.
\end{proof}

\subsubsection{Edges Connected to Boundary}
    \label{subsec: Proof of The Exponential Decay1}
    For an edge connected to the boundary, i.e. $e \cap \partial \Omega \neq \emptyset$, we have a different geometric relation between $e$ and $\omega_e$; see the right one in Figure \ref{fig:omega omega star}. For this edge, we have that $e\cap \partial \Omega=\{p_e\}$ is a single point. There exists two rectangles $\omega \subset \omega^*$, such that $e \subset \omega \subset \omega^* \subset \omega_e$. One side of $\omega$ and $\omega^*$ is parallel to $e$. We require (1) $\partial \omega^*$ only intersects with $\partial \Omega$ on one of the four edges of the latter; (2) $p_e$ is the center of gravity of $\partial \omega \cap \partial \Omega$ and $\partial \omega^* \cap \partial \Omega$. We introduce three parameters $l_4,l_5,l_6$ to specify and describe the geometry:
    \begin{enumerate}
        \item With respect to the center $p_e$, the two rectangles $\omega$ and $\omega^*$ are scaling equivalent, such that there exists $l_4>1$, $\omega^*-p_e=l_4\cdot (\omega-p_e)$. For our choice of $\omega_e$, the parameter $l_4$ can be selected to only depend on $c_0$ and $c_1$ in Subsection \ref{subsec:elements}.
        \item The ratio of $\omega$'s larger side length over the smaller side length is bounded by a uniform constant $l_5>1$ that depends on $c_0$ and $c_1$ only.
        \item There is a constant $l_6 > 1$ depending on $c_0$ and $c_1$ only such that $l_6\cdot e \subset \omega$.
    \end{enumerate}
      We note that $l_4,l_5,l_6$ are universal constants for all edges connected to the boundary. All three parameters depend on $c_0,c_1$ only.
      
    As in the last section, we have two lemmas that are needed in the proof of the theorem. The first lemma is the same as Lemma \ref{conjecture1}, which also applies for edges connected to the boundary due to the global H\"older estimate in Theorem 8.29 of \cite{gilbarg2015elliptic}. The second lemma is similar to Lemma \ref{thm: decay of restriction operator for harmonic functions}, while here we define $U(\omega^*)$ and $U(\omega)$ to be the spaces of $a$-harmonic functions in $\omega^*$ and $\omega$ that vanish at $\omega^*\cap \partial \Omega$ and $\omega\cap \partial \Omega$, respectively. The restriction operator follows the same definition as \eqref{eqn: oversampling restriction operator}. Regarding its singular values, we have the following lemma:
    \begin{lemma}
        \label{thm: decay of restriction operator for harmonic functions, edge connected to boundary}
       For each edge $e$ connected to boundary, the operator $P_{\omega}$ is bounded and compact.  Let the pairs of left singular vectors and singular values of $P_{\omega}$ be $\{v_{e,m},\mu_{e,m} \}_{m \in \bN}$ and the sequence $\{\mu_{e,m}\}_{m \in \bN}$ is in a descending order.  Then for any $\epsilon>0$, there exists an $N_{\epsilon}>0$, such that for all $m > N_{\epsilon}$, it holds that
     \begin{equation}
         \mu_{e,m}\leq \exp\left(-m^{(\frac{1}{d+1}-\epsilon)}\right)\, ,
     \end{equation}
     where $N_{\epsilon}$ depends on $d, a_{\min}, a_{\max}, c_0$ and $c_1$.
    
     Therefore, if we set $\Phi_{e,m}=\mathrm{span}~\{v_{e,k}\}_{k=1}^{m-1}$ for some $m>N_{\epsilon}$, then we have
    \begin{equation}
      \min_{\chi \in \Phi_{e,m}}\|P_\omega v-\chi\|_{H_a^1(\omega)}\leq \exp\left(-m^{(\frac{1}{d+1}-\epsilon)}\right) \|v\|_{H_a^1(\omega^*)}\, ,
    \end{equation}
    for any $v \in (U(\omega^*), \|\cdot\|_{H_a^1(\omega^*)})$.    
    \end{lemma}
    Lemma \ref{thm: decay of restriction operator for harmonic functions, edge connected to boundary} is a generalization of Theorem 3.7 in \cite{babuska2011optimal} into cases where $\omega$ is no longer a cube but a shape regular rectangle and the boundary condition is of Dirichlet's type. The result remains valid and the proof remains the same; the key steps are Caccioppoli inequality and iteration arguments, which lead to a quantitative estimate of this compact embedding. Combining Lemma \ref{conjecture1} and Lemma \ref{thm: decay of restriction operator for harmonic functions, edge connected to boundary} concludes the proof for the boundary part; the argument is the same as that in the proof for the interior edges.
    \section{Numerical Experiments}
    \label{sec: Numerical experiments}
    In this section, we conduct numerical experiments to validate our theoretical analysis and demonstrate the effectiveness of our method.  We will test three representative examples of the coefficient $a(x)$: (1) with multiple scales; (2) drawn from random Gaussian field with no scale separation; (3) with high contrast. 
     \subsection{Three Choices in the Algorithm} \label{subsec: Three Choices in the Algorithm} There can be different choices in our algorithm, depending on whether the bubble part $u^\sfb$ is computed, and whether the oversampling bubble part $P_e(u_{\omega_e}^{\sfb}-I_Hu_{\omega_e}^{\sfb})$ is incorporated into the edge basis functions. We mainly focus on the following three choices. Here, we use $H$ to denote the coarse mesh size and $m$ is the number of left singular vectors of $R_e$ that are used in constructing edge basis functions. We choose $m$ to be the same for every edge $e \in \cE_H$.
    \begin{enumerate}
        \item In the first choice, we compute the Galerkin solution without using $P_e(u_{\omega_e}^{\sfb}-I_Hu_{\omega_e}^{\sfb})$ in the edge basis functions. Moreover, $u^\sfb$ is not computed. The solution depends on $H$ and $m$, and is denoted by $u_{H,m}^{(1)}$. Theoretically, we expect $m=O(\log^{d+1}(1/H))$ leads to $O(H)$ solution accuracy in the energy norm; see Remark \ref{rmk: no bubble and os bubble}. 
        \item In the second choice, we further compute the bubble part $u^\sfb$, and add it to $u_{H,m}^{(1)}$ to get $u_{H,m}^{(2)}=u_{H,m}^{(1)}+u^\sfb$. This reveals one of the benefits of our energy orthogonal design, namely the fine scale component $u^\sfb$ is easy to compute. Our theory in Remark \ref{rmk: with or without bubble} implies that we would still get the same rate as $u_{H,m}^{(1)}$ in the upper bound of the accuracy. However, we will demonstrate that $u_{H,m}^{(2)}$ does lead to much better numerical accuracy compared to $u_{H,m}^{(1)}$, while the additional computational complexity is negligible.
        \item In the third choice, we compute $P_e(u_{\omega_e}^{\sfb}-I_Hu_{\omega_e}^{\sfb})$ and incorporate it into the edge basis functions on $e$. We also compute the bubble part $u^\sfb$ and add it to the Galerkin solution. The final solution is denoted by $u_{H,m}^{(3)}$. Our theory in Remark \ref{rmk: with or without bubble} implies that we should get a nearly exponentially decaying approximation error with respect to $m$. We test its numerical performance when different kinds of $a(x)$ are present.
    \end{enumerate}
    In all the numerical examples, we consider the elliptic problem with homogeneous boundary condition in the domain $\Omega=[0,1]\times [0,1]$. 
    \begin{equation}
    \left\{
    \begin{aligned}
    -\nabla \cdot (a \nabla u )&=f, \quad \text{in} \  \Omega\\
    u&=0, \quad \text{on} \  \partial\Omega\, .
    \end{aligned}
    \right.
    \end{equation}
    We discretize the domain by a uniform two-level quadrilateral mesh; see a fraction of this mesh in Figure \ref{fig:mesh1}, where we also show an edge $e$ and its oversampling domain $\omega_e$ in solid lines. 

\begin{figure}[htbp]
\centering

\tikzset{every picture/.style={line width=0.75pt}} 

\begin{tikzpicture}[x=0.75pt,y=0.75pt,yscale=-1,xscale=1]

\draw  [draw opacity=0][dash pattern={on 4.5pt off 4.5pt}] (170.5,75) -- (365.5,75) -- (365.5,231) -- (170.5,231) -- cycle ; \draw  [color={rgb, 255:red, 0; green, 0; blue, 0 }  ,draw opacity=1 ][dash pattern={on 4.5pt off 4.5pt}] (209.5,75) -- (209.5,231)(248.5,75) -- (248.5,231)(287.5,75) -- (287.5,231)(326.5,75) -- (326.5,231) ; \draw  [color={rgb, 255:red, 0; green, 0; blue, 0 }  ,draw opacity=1 ][dash pattern={on 4.5pt off 4.5pt}] (170.5,114) -- (365.5,114)(170.5,153) -- (365.5,153)(170.5,192) -- (365.5,192) ; \draw  [color={rgb, 255:red, 0; green, 0; blue, 0 }  ,draw opacity=1 ][dash pattern={on 4.5pt off 4.5pt}] (170.5,75) -- (365.5,75) -- (365.5,231) -- (170.5,231) -- cycle ;
\draw  [draw opacity=0][dash pattern={on 0.84pt off 2.51pt}] (170.5,75) -- (365.5,75) -- (365.5,231) -- (170.5,231) -- cycle ; \draw  [color={rgb, 255:red, 0; green, 0; blue, 0 }  ,draw opacity=1 ][dash pattern={on 0.84pt off 2.51pt}] (183.5,75) -- (183.5,231)(196.5,75) -- (196.5,231)(209.5,75) -- (209.5,231)(222.5,75) -- (222.5,231)(235.5,75) -- (235.5,231)(248.5,75) -- (248.5,231)(261.5,75) -- (261.5,231)(274.5,75) -- (274.5,231)(287.5,75) -- (287.5,231)(300.5,75) -- (300.5,231)(313.5,75) -- (313.5,231)(326.5,75) -- (326.5,231)(339.5,75) -- (339.5,231)(352.5,75) -- (352.5,231) ; \draw  [color={rgb, 255:red, 0; green, 0; blue, 0 }  ,draw opacity=1 ][dash pattern={on 0.84pt off 2.51pt}] (170.5,88) -- (365.5,88)(170.5,101) -- (365.5,101)(170.5,114) -- (365.5,114)(170.5,127) -- (365.5,127)(170.5,140) -- (365.5,140)(170.5,153) -- (365.5,153)(170.5,166) -- (365.5,166)(170.5,179) -- (365.5,179)(170.5,192) -- (365.5,192)(170.5,205) -- (365.5,205)(170.5,218) -- (365.5,218) ; \draw  [color={rgb, 255:red, 0; green, 0; blue, 0 }  ,draw opacity=1 ][dash pattern={on 0.84pt off 2.51pt}] (170.5,75) -- (365.5,75) -- (365.5,231) -- (170.5,231) -- cycle ;
\draw    (209.5,114) -- (209.5,192) -- (326.5,192) -- (326.5,114) -- cycle ;

\draw [color={rgb, 255:red, 0; green, 0; blue, 0 }  ,draw opacity=1 ][line width=0.75]    (248.5,153) -- (287.5,153) ;
\draw  [draw opacity=0][dash pattern={on 0.84pt off 2.51pt}] (407.5,157) -- (507.5,157) -- (507.5,218) -- (407.5,218) -- cycle ; \draw  [color={rgb, 255:red, 0; green, 0; blue, 0 }  ,draw opacity=1 ][dash pattern={on 0.84pt off 2.51pt}] (417.5,157) -- (417.5,218)(430.5,157) -- (430.5,218)(443.5,157) -- (443.5,218)(456.5,157) -- (456.5,218)(469.5,157) -- (469.5,218)(482.5,157) -- (482.5,218)(495.5,157) -- (495.5,218) ; \draw  [color={rgb, 255:red, 0; green, 0; blue, 0 }  ,draw opacity=1 ][dash pattern={on 0.84pt off 2.51pt}] (407.5,167) -- (507.5,167)(407.5,180) -- (507.5,180)(407.5,193) -- (507.5,193)(407.5,206) -- (507.5,206) ; \draw  [color={rgb, 255:red, 0; green, 0; blue, 0 }  ,draw opacity=1 ][dash pattern={on 0.84pt off 2.51pt}]  ;
\draw  [draw opacity=0][dash pattern={on 4.5pt off 4.5pt}] (406.5,73) -- (508.5,73) -- (508.5,135) -- (406.5,135) -- cycle ; \draw  [color={rgb, 255:red, 0; green, 0; blue, 0 }  ,draw opacity=1 ][dash pattern={on 4.5pt off 4.5pt}] (416.5,73) -- (416.5,135)(455.5,73) -- (455.5,135)(494.5,73) -- (494.5,135) ; \draw  [color={rgb, 255:red, 0; green, 0; blue, 0 }  ,draw opacity=1 ][dash pattern={on 4.5pt off 4.5pt}] (406.5,83) -- (508.5,83)(406.5,122) -- (508.5,122) ; \draw  [color={rgb, 255:red, 0; green, 0; blue, 0 }  ,draw opacity=1 ][dash pattern={on 4.5pt off 4.5pt}]  ;

\draw (263.5,140) node [anchor=north west][inner sep=0.75pt]   [align=left] {$\displaystyle e$};
\draw (303.5,115) node [anchor=north west][inner sep=0.75pt]   [align=left] {$\displaystyle \omega _{e}$};
\draw (414,136) node [anchor=north west][inner sep=0.75pt]   [align=left] {coarse mesh};
\draw (424,218) node [anchor=north west][inner sep=0.75pt]   [align=left] {fine mesh};

\end{tikzpicture}

\caption{Two level mesh: a fraction}
\label{fig:mesh1}
\end{figure}
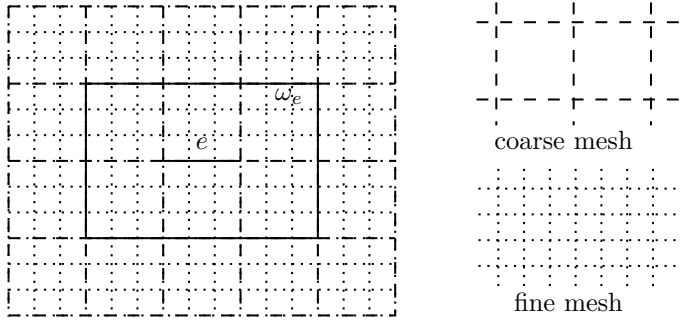
    The fine mesh is of size $h=1/1024$, which is fine enough to resolve the fine scale information of $a(x)$. The reference solution $u_{\text{ref}}$ is computed using the classical FEM on the fine mesh. Since $h$ is small, we will treat $u_{\text{ref}}$ as the ground truth $u$. The accuracy of our solutions ($u^{(k)}_{H,m}, k=1,2,3$) is computed by comparing them with the reference solution $u_{\text{ref}}$ on the fine mesh. The accuracy will be measured both in the $L^2$ norm and energy norm. They depend on the mesh size $H$, the number of enrichment basis functions $m$, the coefficient $a$ and the right-hand side $f$. We use the notation below:
    \begin{equation}
\label{rel_error}
\begin{aligned}
e^{(k)}_{L^{2}}(H,m, a, f)&=\frac{\|u_{\text{ref}}-u_{H,m}^{(k)}\|_{L^{2}(\Omega)}}{\|u_{\text{ref}}\|_{L^{2}(\Omega)}}\, ,\\
e^{(k)}_{E}(H,m, a, f)&=\frac{\|u_{\text{ref}}-u_{H,m}^{(k)}\|_{H^1_a(\Omega)}}{\|u_{\text{ref}}\|_{H^1_a(\Omega)}}\, . 
\end{aligned}
\end{equation}
   We will vary these parameters to test the convergence behavior of our method. We start with a detailed implementation of the algorithm in the next subsection.
    
\subsection{Implementation of Our Algorithm} In this subsection, we outline the implementation of our algorithm for computing $u^{(k)}_{H,m}, k=1,2,3$. 
We consider the setting that for a given $a$ we want to solve the equation with multiple $f$.  
In such a scenario, for the sake of efficiency, steps that are independent of $f$ will be computed offline and only be implemented once. Steps that are adaptive to $f$ will be put online and need to be implemented every time with a new $f$.
\subsubsection{Offline Stage} We start with the offline stage. The steps are outlined in order below. 
    \begin{enumerate}
    \item For each element of the coarse mesh $T \in \cT_H$, we build the local stiffness matrix for the local elliptic problem with Dirichlet's boundary condition.
    We can use this matrix to solve the local elliptic equation with arbitrary Dirichlet's boundary data effectively. Specifically, we can do $a$-harmonic extension of functions on edges to elements. 
    \item For each oversampling domain $\omega_e$, we build the local stiffness matrix for the elliptic problem in $\omega_e$ with Dirichlet's boundary condition, repeating the last step with $T$ replaced by $\omega_e$. 
    \item For each edge $e$ and its oversampling domain $\omega_e$, we build the matrix version of the operator $R_e$. Theoretically, the domain of $R_e$ comprises $a$-harmonic functions in $\omega_e$; they are fully determined by their trace on $\partial \omega_e$. Numerically, our matrix version of $R_e$, when viewed as a linear mapping, maps Dirichlet's boundary data on $\partial \omega_e$ to the image of $R_e$ on $e$.  
    \item After we discretize the corresponding norms in the domain and image of $R_e$, the SVD problem for $R_e$ transforms to a generalized eigenvalue problem for the discrete matrices. We use standard algorithms (in this paper, we use MATLAB's default $\mathsf{eig(A,B)}$) to get the top-$m$ singular vectors. Together with the interpolation part, their $a$-harmonic extensions constitute the basis functions that will be used in the Galerkin method.
    
    
    \item Using the basis functions, we assembly the corresponding global stiffness matrix by running over each element. 
        \end{enumerate}
       
        \begin{remark}
        The main computational cost in the offine line is step 4: the construction of edge basis functions. It requires us to compute the first $m$ singular vectors of $R_e$ corresponding to the $m$ largest singular values. Any fast algorithm of SVD can be employed here. In our implementation, we have used the simple MATLAB's default functions. It is also possible to use tools in randomized numerical linear algebra, such as the randonmized SVD to obtain efficient computation of these $m$ most significant singular vectors; see for example the works \cite{buhr2018randomized, chen2020randomized}. As mentioned in Remark \ref{rmk: spatial adaptivity}, we can also choose a different number of edge basis functions, say $m_e$, for each edge $e$. The final accuracy of the solution $u_{H,m}^{(3)}$ will be of order $(\max_{e\in \cE_H} \lambda_{e,m_e})\|f\|_{L^2(\Omega)}$ where $\lambda_{e,m_e}$ is the $m_e$-th singular value of $R_e$. We have an analytic bound
        \[\lambda_{e,m_e}\leq C\exp\left(-m_e^{(\frac{1}{d+1}-\epsilon)}\right)\]
        due to Theorem \ref{thm: svd exponential decay of Re}, so $O(\log^{d+1}(1/H))$ number of basis functions for each edge would be sufficient to achieve $O(H)$ accuracy. Nevertheless, in our numerical experiments, we found that a small number of basis functions have shown very good accuracy already. The nearly exponential decay is also observed in all our numerical experiments.
        \end{remark}
    
    \subsubsection{Online Stage} The online stages consists of the following steps. We put a bracket at the end of each step to indicate for which $k$ this step is needed.
    \begin{enumerate}
    \item Compute the bubble part $u^\sfb$ using the local stiffness matrix in each $T$ built offline. ($k=2,3$)
    \item Compute the oversampling bubble part $u^\sfb_{\omega_e}$ using the local stiffness matrix in each $\omega_e$ built offline. Use the boundary data $P_e ( u_{\omega_e}^{\sfb}-I_H u_{\omega_e}^{\sfb})$ on each $e$ to perform the $a$-harmonic extension; this leads to a basis function adaptive to $f$ for each $e$. Combine these new basis functions with those built offline to assembly an updated global stiffness matrix. Since we have computed an old stiffness matrix for the offline basis functions, we only need to incorporate the additional parts introduced by the new basis functions. ($k=3$) 
    \item Compute the Galerkin solution using $f$ and the global stiffness matrix. ($k=1,2,3$)
    \item Add the bubble part computed in the first step to the Galerkin solution. ($k=2,3$)
        \end{enumerate}
 We test our algorithm for three examples in the next subsections.      
\subsection{An Example with Multiple Spatial Scales}
    In the first example, we choose $a(x)$ with five scales as follows: 
   \begin{equation}
\begin{aligned} 
a(x)=\frac{1}{6}\left(\frac{1.1+\sin \left(2 \pi x_1 / \epsilon_{1}\right)}{1.1+\sin \left(2 \pi x_2 / \epsilon_{1}\right)}+\frac{1.1+\sin \left(2 \pi x_2 / \epsilon_{2}\right)}{1.1+\cos \left(2 \pi x_1 / \epsilon_{2}\right)}+\frac{1.1+\cos \left(2 \pi x_1 / \epsilon_{3}\right)}{1.1+\sin \left(2 \pi x_2 / \epsilon_{3}\right)}\right.\\\left.+\frac{1.1+\sin \left(2 \pi x_2 / \epsilon_{4}\right)}{1.1+\cos \left(2 \pi x_1 / \epsilon_{4}\right)}+\frac{1.1+\cos \left(2 \pi x_1 / \epsilon_{5}\right)}{1.1+\sin \left(2 \pi x_2 / \epsilon_{5}\right)}+\sin \left(4 x_1^{2} x_2^{2}\right)+1\right) \, ,\end{aligned}
\end{equation}
where $\epsilon_1=1/5$, $\epsilon_2=1/13$, $\epsilon_3=1/17$, $\epsilon_4=1/31$, $\epsilon_5=1/65$. We set $f=-1$.
    
    First, we test the convergence with respect to $H$. We compute $u^{(1)}_{H,m}$ for $m=0,1,2$ and $H=1/2^l$ where $l=3,4,5,6,7$. The energy and $L^2$ errors are shown in Figure \ref{fig: eg1, w.r.t H}.
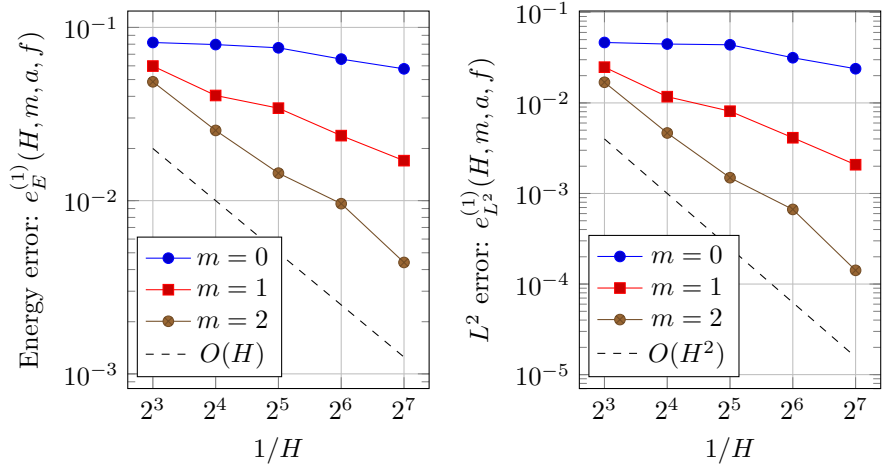
\begin{figure}[thbp]
    \centering
\begin{tikzpicture}
\begin{loglogaxis}[
width=2.2in, height=2.6in,
grid=major, 
xlabel={$1/H$}, 
ylabel={Energy error: $e^{(1)}_{E}(H,m, a, f)$}, 
legend pos=south west,
xtick={2^3,2^4,2^5,2^6,2^7},
xticklabels={$2^3$,$2^4$,$2^5$,$2^6$,$2^7$},
legend entries={$m=0$,$m=1$,$m=2$,$O(H)$}
]
\addplot coordinates {(8,8.1829e-2)(16,7.9607e-2)(32,7.6258e-2)(64,6.5624e-2)(128,5.7652e-2)};
\addplot coordinates {(8,5.9829e-2)(16,4.0457e-2)(32, 3.4185e-2)(64,2.3728e-2)(128,1.7015e-2)}; 
\addplot coordinates {(8,4.8524e-2)(16,2.5424e-2)(32, 1.4414e-2)(64,9.6038e-3)(128,4.3966e-3)}; 
\addplot[dashed] coordinates {(8,2e-2)(16,1e-2)(32, 5e-3)(64,2.5e-3)(128,1.25e-3)}; 
\end{loglogaxis}
\end{tikzpicture}
\hskip 5pt
\begin{tikzpicture}
\begin{loglogaxis}[
width=2.2in, height=2.6in, grid=major, 
xlabel={$1/H$}, 
ylabel={$L^2$ error: $e^{(1)}_{L^2}(H,m, a, f)$}, 
legend pos=south west,
xtick={2^3,2^4,2^5,2^6,2^7},
xticklabels={$2^3$,$2^4$,$2^5$,$2^6$,$2^7$},
legend entries={$m=0$,$m=1$,$m=2$,$O(H^2)$}
]
\addplot coordinates {(8,4.6485e-2)(16,4.4665e-2)(32,4.3847e-2)(64,3.1537e-2)(128,2.3839e-2)};
\addplot coordinates {(8,2.4809e-2)(16,1.1718e-2)(32, 8.0972e-3)(64,4.1258e-3)(128,2.0774e-3)}; 
\addplot coordinates {(8,1.6902e-2)(16,4.6547e-3)(32, 1.4931e-3)(64,6.6684e-4)(128,1.4170e-4)}; 
\addplot[dashed] coordinates {(8,8e-3/2)(16,2e-3/2)(32, 5e-4/2)(64,1.25e-4/2)(128,3.125e-5/2)}; 
\end{loglogaxis}
\end{tikzpicture}
\caption{Example 1. The coefficient $a$ has multiple scales, $f=-1$, $k=1$ and $m=0,1,2$; left: energy error w.r.t. $H$; right: $L^2$ error w.r.t. $H$.}
\label{fig: eg1, w.r.t H}
\end{figure}
We remark that $m=0$ corresponds to the vanilla MsFEM. We observe from Figure \ref{fig: eg1, w.r.t H} that the vanilla MsFEM does not lead to convergence of errors.  However, if we set $m=2$, then we can identify a roughly $O(H)$ tendency in the energy error and an $O(H^2)$ accuracy in the $L^2$ norm. This implies our edge basis functions can correctly capture the multiscale behavior of the solution. In addition, we observe that increasing $m$ may be more efficient than decreasing $H$.
    
Next, we fix $H=1/32$ and vary $m$ from $1$ to $7$ to compute $u^{(k)}_{H,m}$, $k=1,2,3$. The results are illustrated in Figure \ref{fig: eg1, const f, w.r.t m}. 
We observe a desired nearly exponential decay error for $u^{(3)}_{H,m}$. Without using bubble parts, the error for $u^{(1)}_{H,m}$ saturates after $m\geq 3$. Moreover, after adding $u^\sfb$, the solution $u^{(2)}_{H,m}$ achieves a much better accuracy compared to $u^{(1)}_{H,m}$. This demonstrates the benefits of our energy orthogonal design. Even without the oversampling bubble part, we can still get very accurate approximation. Because the additional computational cost for $u^{(2)}_{H,m}$ compared to $u^{(1)}_{H,m}$ is negligible, this result tells us that we should always compute the bubble part if we could. In the subsequent examples, we will omit the case $k=1$, and mainly investigate the performance of $k=2,3$. 
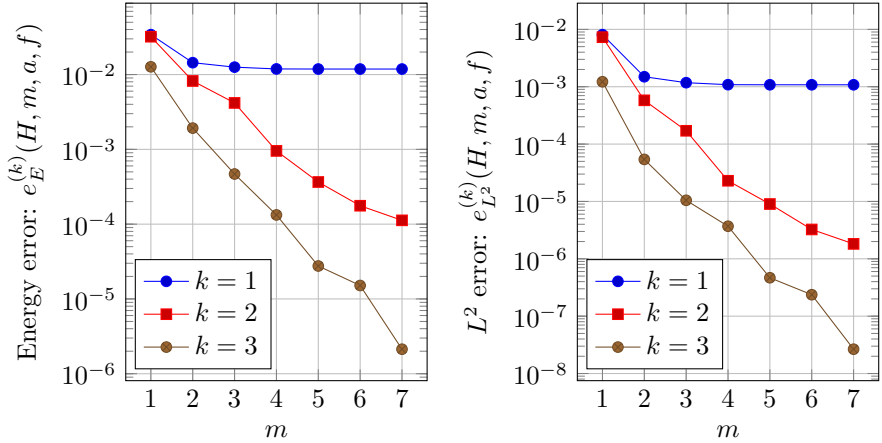
\begin{figure}[thbp]
    \centering
\begin{tikzpicture}
\begin{semilogyaxis}[
width=2.2in, height=2.6in,
grid=major, 
xlabel={$m$}, 
ylabel={Energy error: $e^{(k)}_{E}(H,m, a, f)$}, 
legend pos=south west,
xtick={1,2,3,4,5,6,7},
xticklabels={$1$,$2$,$3$,$4$,$5$,$6$,$7$},
legend entries={$k=1$,$k=2$,$k=3$}
]
\addplot coordinates {
   (1,3.4185e-02)
   (2,1.4414e-02)
   (3,1.2563e-02)
   (4,1.1887e-02)
   (5,1.1854e-02)
   (6,1.1850e-02)
   (7,1.1849e-02)};
\addplot coordinates {
   (1,3.2066e-02)
   (2,8.2077e-03)
   (3,4.1744e-03)
   (4,9.5310e-04)
   (5,3.6588e-04)
   (6,1.7630e-04)
   (7,1.1240e-04)
}; 
\addplot coordinates {
   (1,1.2692e-02)
   (2,1.9230e-03)
   (3,4.6793e-04)
   (4,1.3273e-04)
   (5,2.7569e-05)
   (6,1.5128e-05)
   (7,2.1279e-06)
}; 
\end{semilogyaxis}
\end{tikzpicture}
\hskip 5pt
\begin{tikzpicture}
\begin{semilogyaxis}[
width=2.2in, height=2.6in, grid=major, 
xlabel={$m$}, 
ylabel={$L^2$ error: $e^{(k)}_{L^2}(H,m, a, f)$}, 
legend pos=south west,
xtick={1,2,3,4,5,6,7},
xticklabels={$1$,$2$,$3$,$4$,$5$,$6$,$7$},
legend entries={$k=1$,$k=2$,$k=3$}
]
\addplot coordinates {
(1,8.0972e-03)
   (2,1.4931e-03)
   (3,1.1763e-03)
   (4,1.0854e-03)
   (5,1.0817e-03)
   (6,1.0813e-03)
   (7,1.0812e-03)
}; 
\addplot coordinates {
(1,7.3040e-03)
   (2,5.7978e-04)
   (3,1.7068e-04)
   (4,2.2953e-05)
   (5,9.0171e-06)
   (6,3.2485e-06)
   (7,1.8206e-06)
}; 
\addplot coordinates {
(1,1.2209e-03)
   (2,5.4079e-05)
   (3,1.0471e-05)
   (4,3.6988e-06)
   (5,4.6740e-07)
   (6,2.3659e-07)
   (7,2.6555e-08)
};
\end{semilogyaxis}
\end{tikzpicture}
\caption{Example 1. The coefficient $a$ has multiple scales, $f=-1$, $H=1/32$, and $k=1,2,3$; left: energy error w.r.t. $m$; right: $L^2$ error w.r.t. $m$.}
\label{fig: eg1, const f, w.r.t m}
\end{figure}

We remark that in our uniform coarse mesh, there are $2/H\cdot (1/H-1)$ coarse edges and $(1/H-1)^2$ interior nodes, if $1/H$ is an integer. Thus, the number of basis functions used for computing $u^{(1)}_{H,m}, u^{(2)}_{H,m}$ is $2/H\cdot (1/H-1)m+(1/H-1)^2$, while for $u^{(3)}_{H,m}$ it is $2/H\cdot (1/H-1)(m+1)+(1/H-1)^2$. Among all these basis functions, $(1/H-1)^2$ of them belong to the interpolation part and have supports in the union of four elements, while the remaining ones are supported in the union of two elements.

    
        \subsection{Random Field without Scale Separation}
        In the second example, we choose $a(x)$ to be a realization of some random field. More precisely, we set
        \begin{equation}
a(x)=|\xi(x)|+0.5\, ,
\end{equation}
        where the field $\xi(x)$ satisfies \[\xi(x)=a_{11}\xi_{i,j}+a_{21}\xi_{i+1,j}+a_{12}\xi_{i,j+1}+a_{22}\xi_{i+1,j+1}, \quad \text{if}\ x \in [\frac{i}{2^{7}},\frac{i+1}{2^{7}})\times [\frac{j}{2^{7}},\frac{j+1}{2^{7}})\, .\]
        Here, $\{\xi_{i,j}, 0\leq i,j \leq 2^7 \}$ are i.i.d. unit Gaussian random variables. In addition, $a_{11}=(i+1-2^7x_1)(j+1-2^7x_2)$, $a_{21}=(2^7x_1-i)(j+1-2^7x_2)$, $a_{12}=(i+1-2^7x_1)(2^7x_2-j)$, $a_{22}=(2^7x_1-i)(2^7x_2-j)$ are interpolating coefficients to make $\xi(x)$ piecewise linear.
        Generally, a typical realization of the field $a(x)$ is rough and of no scale separation. 

    We choose the same right-hand side $f=-1$. For a single realization of $a(x)$, we compute $u^{(k)}_{H,m}$ for $k=2,3$. As before, we fix $H=1/32$ and vary $m$ from $1$ to $7$. The errors are output in Figure \ref{fig: eg2, const f, w.r.t m}. For this example, we observe that both $k=2$ and $k=3$ yield nearly exponential decay of approximation errors with respect to $m$. Using $k=2$ can already be very efficient, although theoretically we only get an $O(H)$ upper bound for the accuracy.
\begin{figure}[thbp]
    \centering
\begin{tikzpicture}
\begin{semilogyaxis}[
width=2.2in, height=2.6in,
grid=major, 
xlabel={$m$}, 
ylabel={Energy error: $e^{(k)}_{E}(H,m, a, f)$}, 
legend pos=south west,
xtick={1,2,3,4,5,6,7},
xticklabels={$1$,$2$,$3$,$4$,$5$,$6$,$7$},
legend entries={$k=2$,$k=3$}
]
\addplot table[x=m,y=error] {
m error
1	0.0276744954392036
2	0.00924211085077798
3	0.00328549037580530
4	0.000700571079816568
5	0.000216368539312276
6	8.05312104398087e-05
7	5.60600184635196e-05
};
\addplot table[x=m,y=error] {
m error
1	0.0212500633217292
2	0.00666389340574544
3	0.00214562124749342
4	0.000488553064796399
5	0.000115063831227054
6	2.89237594675568e-05
7	1.49393397808907e-05
};
\end{semilogyaxis}
\end{tikzpicture}
\hskip 5pt
\begin{tikzpicture}
\begin{semilogyaxis}[
width=2.2in, height=2.6in, grid=major, 
xlabel={$m$}, 
ylabel={$L^2$ error: $e^{(k)}_{L^2}(H,m, a, f)$}, 
legend pos=south west,
xtick={1,2,3,4,5,6,7},
xticklabels={$1$,$2$,$3$,$4$,$5$,$6$,$7$},
legend entries={$k=2$,$k=3$}
]
\addplot table[x=m,y=error] {
m error
1	0.00470024318363872
2	0.000581116523436520
3	9.62973588489628e-05
4	1.20810556414430e-05
5	3.45040478318727e-06
6	1.14051816487618e-06
7	6.51185906872642e-07
};
\addplot table[x=m,y=error] {
m error
1	0.00282110065492825
2	0.000318266670675074
3	4.78914298312000e-05
4	7.67793925649996e-06
5	1.85796885181761e-06
6	3.97908541223471e-07
7	1.63575981693026e-07
};
\end{semilogyaxis}
\end{tikzpicture}
\caption{Example 1. The coefficient $a$ is a rough random field, $f=-1$, $H=1/32$, and $k=2,3$; left: energy error w.r.t. $m$; right: $L^2$ error w.r.t. $m$.}
\label{fig: eg2, const f, w.r.t m}
\end{figure}
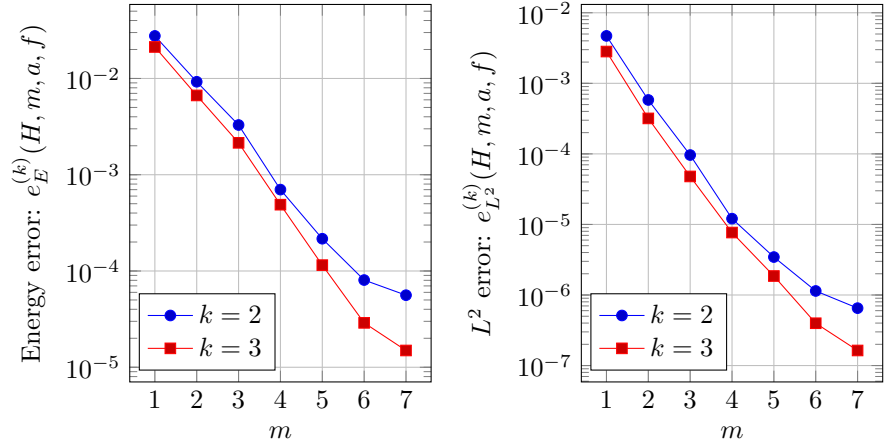

\subsection{An Example with High Contrast Channels}
In the third example, we consider an $a(x)$ with high contrast channels. Let 
\[X:=\{(x_1,x_2) \in [0,1]^2, x_1,x_2 \in \{0.2,0.3,...,0.8\}\} \subset [0,1]^2 \, , \]
and the coefficient is defined as
\begin{equation*}
    a(x)=\left\{ 
    \begin{aligned} 1, \quad &\text{if} \ \dist(x,X)\geq 0.025\\
    M, \quad &\text{else}\, .
    \end{aligned}
    \right.
\end{equation*}
Here, $M$ is a parameter controlling the contrast. We visualize $\log_{10} a(x)$ in Figure $\ref{fig:log a(x)}$ for $M=10^4$.
\begin{figure}[htbp]
    \centering
    \includegraphics[width=8cm]{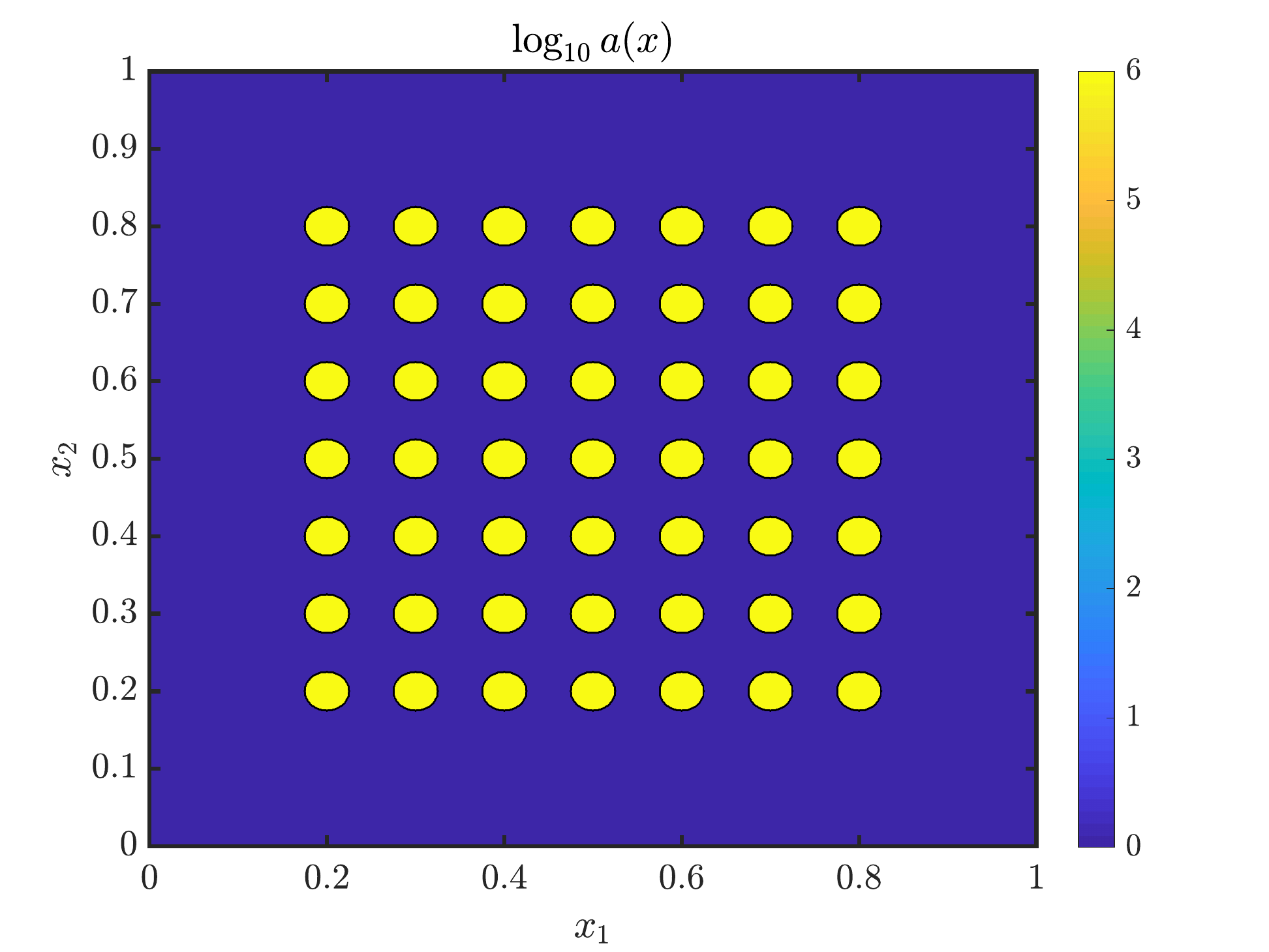}
    \caption{Example 3. The coefficient has high contrasts. The contour of $\log_{10} a(x)$ for $M=10^4$.}
    \label{fig:log a(x)}
\end{figure}
    
    For this high contrast media, we test the performance of $u^{(k)}_{H,m}$ for $k=2,3$ in the following. We use a non-constant right-hand side $f(x)=x_1^4-x_2^3+1$.
    The contrast $M$ is set to be $2^{10},2^{14}$ respectively.  The coarse mesh size is $H=1/32$ and we vary $m=1,2,...,7$. We present the energy errors and $L^2$ errors of the solution in Figures \ref{fig: eg3, energy w.r.t m} and \ref{fig: eg3, L2 w.r.t m}. 
\begin{figure}[thbp]
\centering
\begin{tikzpicture}
\begin{semilogyaxis}[
width=3in, height=2.8in,
grid=major, 
xlabel={$m$}, 
ylabel={Energy error: $e^{(k)}_{E}(H,m, a, f)$}, 
legend pos=north east,
xtick={1,2,3,4,5,6,7},
xticklabels={$1$,$2$,$3$,$4$,$5$,$6$,$7$},
legend entries={$k=2;M=2^{10}$,$k=3;M=2^{10}$, $k=2;M=2^{14}$,$k=3;M=2^{14}$}
]
\addplot table[x=m,y=error] {
m error
1	0.0146611012526805
2	0.00153141746090508
3	0.000502410202338436
4	0.000218123599422783
5	9.99879866531240e-05
6	7.12734875216514e-05
7	5.35156692319484e-05
};
\addplot table[x=m,y=error] {
m error
1	0.00394726646849282
2	0.000609007221641966
3	0.000267135220918456
4	0.000130272159807889
5	8.85113340230787e-06
6	6.05903155839489e-06
7	4.02567413687105e-06
};
\addplot table[x=m,y=error] {
m error
1	0.0147573706985741
2	0.00151432531359210
3	0.000501499499158390
4	0.000213997195011111
5	9.94036042858331e-05
6	7.07719830180688e-05
7	5.34406682037078e-05
};
\addplot table[x=m,y=error] {
m error
1	0.00397138585688078
2	0.000602497093174641
3	0.000263912500789952
4	0.000124975956061889
5	7.85163856769751e-06
6	5.29418711147646e-06
7	4.11733694222059e-06
};
\end{semilogyaxis}
\end{tikzpicture}
\caption{Example 3. The coefficient $a$ has high contrast, $f(x)=x_1^4-x_2^3+1$, $H=1/32$, and $k=2,3$. The energy error w.r.t. $m$.}
\label{fig: eg3, energy w.r.t m}
\end{figure}
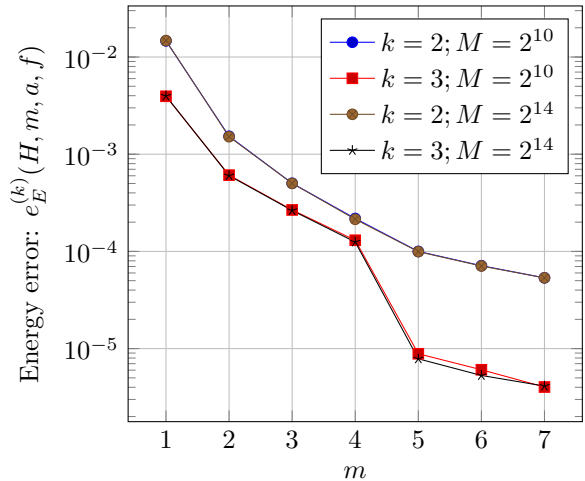
\begin{figure}[thbp]
\centering
\begin{tikzpicture}
\begin{semilogyaxis}[
width=3in, height=2.8in,
grid=major, 
xlabel={$m$}, 
ylabel={$L^2$ error: $e^{(k)}_{L^2}(H,m, a, f)$}, 
legend pos=north east,
xtick={1,2,3,4,5,6,7},
xticklabels={$1$,$2$,$3$,$4$,$5$,$6$,$7$},
legend entries={$k=2;M=2^{10}$,$k=3;M=2^{10}$, $k=2;M=2^{14}$,$k=3;M=2^{14}$}
]
\addplot table[x=m,y=error] {
m error
1	0.00191324327878729
2	4.85458593617026e-05
3	9.98804157286652e-06
4	4.29671381502487e-06
5	1.47778922169440e-06
6	9.07747800276336e-07
7	5.99293456645945e-07
};
\addplot table[x=m,y=error] {
m error
1	0.000188355565270131
2	1.30153519818250e-05
3	4.42745646572276e-06
4	2.79362019611027e-06
5	9.04014582138583e-08
6	5.92727590344307e-08
7	3.77143964637840e-08
};
\addplot table[x=m,y=error] {
m error
1	0.00193918209258177
2	4.84257331483478e-05
3	1.00058821756274e-05
4	4.32849026527155e-06
5	1.47654461014832e-06
6	9.01917435085700e-07
7	5.99486031908748e-07
};
\addplot table[x=m,y=error] {
m error
1	0.000191560044691069
2	1.30657461164631e-05
3	4.44681370202308e-06
4	2.81740963239875e-06
5	8.99320064722522e-08
6	5.99321619562078e-08
7	3.90425865804247e-08
};
\end{semilogyaxis}
\end{tikzpicture}
\caption{Example 3. The coefficient $a$ is has high contrast, $f(x)=x_1^4-x_2^3+1$, $H=1/32$, and $k=2,3$. The $L^2$ error w.r.t. $m$.}
\label{fig: eg3, L2 w.r.t m}
\end{figure}
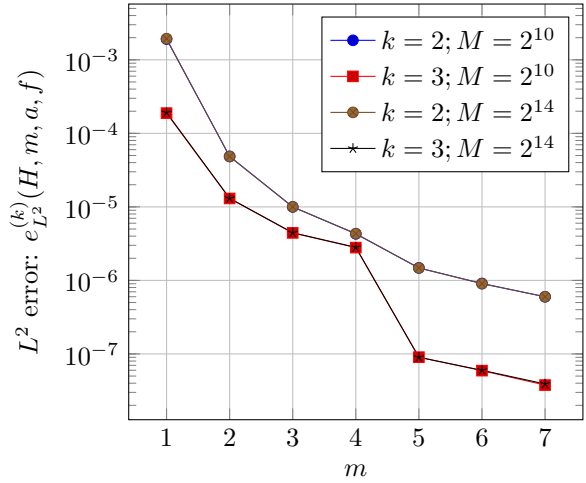    

We observe a contrast independent error decay: our method demonstrates some robustness with respect to high contrast in $a(x)$. We believe the reason could be that every step in the algorithm is adaptive to $a(x)$, for example, the singular value decay of the operator $R_e$ would have some robustness regarding high contrasts in $a(x)$ because both of the norms in the domain and image of this operator is $a(x)$-weighted. We leave the theoretical analysis of deriving an $a(x)$-adapted estimates for future study. 
\section{Concluding Remarks} In this section, we summarize the main findings of this paper, provide several relevant discussions and generalizations, and draw our conclusions accordingly.
\label{sec: discussion}
\subsection{Summary}
In this paper, we developed a multiscale framework to solve second-order linear elliptic PDEs with rough coefficients, extending the previous work \cite{hou2015optimal}. The energy orthogonal decomposition of the solution space into $ a $-harmonic parts and bubble parts is the critical component of our approach, leading to distinct treatments for the two components. The bubble part depends on local information of $f$ and can be computed locally and efficiently. The $ a $-harmonic part depends entirely on function values on edges, so we use multiscale edge basis functions to approximate it. The initial decomposition, combined with a subsequent coupling of the two components, eventually leads to nearly exponential convergence of the approximation error with respect to degrees of freedom. We provided a rigorous proof of this phenomenon. 

\subsection{Discussions}
The most technical part of the proof lies in the approximation of the coarse scale component. Our coarse scale function $u^\sfh$, being $a$-harmonic in each local element, can be approximated with exponential efficiency if some information from the oversampling domains is utilized. We used a result from \cite{babuska2011optimal}, which is Lemma \ref{thm: decay of restriction operator for harmonic functions} in this paper, to prove this exponential accuracy rigorously. Our numerical experiments validated the theoretical analysis. We observed evidence of a nearly exponential convergence. Compared to the work \cite{babuska2011optimal}, which can also achieved exponential convergence and relies on a partition of unity for the overlapping domain decomposition, our edge coupling uses a non-overlapping decomposition of the domain. Thus, our local domain and local basis functions would have a smaller length scale, and every single local computation is more efficient. On the other hand, the number of edges is larger than the number of elements so we may get more basis functions. It is of future interest to compare our approach with the one in \cite{babuska2011optimal} and to establish a more systematic comparison of algorithms with overlapping and non-overlapping domains. 
Moreover, both \cite{babuska2011optimal} and our work deal with the level of multiscale model reduction. It would be interesting to explore the multilevel structure of the coarse scale space $V^{\sfh}$, possibly via a multigrid type of algorithm. The work of Gamblets \cite{owhadi_multigrid_2017, owhadi2019operator} has achieved this goal, for a different energy orthogonal decomposition of the solution space.

    Our current theory does not fully explain the robustness with respect to the contrast of $a(x)$ that we observed in the experiments. It is of future interest to perform a theoretical analysis of this phenomenon. Intuitively, a guiding principle for designing methods robust to high contrast may be to make every step $a(x)$-adaptive. Our approach is indeed $a(x)$-adaptive. The $\cH^{1/2}(e)$ norm and the $H_a^1(\omega_e)$ norm used in the SVD both depend on $a(x)$. We expect an estimate of the decay of singular values that has better scaling regarding the contrast of $ a(x)$.
    
    In Theorem \ref{thm: Edge coupling error estimate}, the integration of local errors to a global error implies the importance of performing approximation on each edge in the $H^{1/2}_{00}(e)$ space. This is one of the main messages that we would like to convey in this paper. We also believe this theorem may lead to broader applications; it provides a general way of coupling errors through local edge approximation. Various ways of local edge approximation can be designed based on this theorem, and the final accuracy will depend on the decay of the singular values of some related operator, as we have discussed in Subsection \ref{sec: Connection to Existing Works}.
    
    \subsection{Generalizations of the Method}
    There can be several natural generalizations of the method in this paper. The first is on the boundary condition. When the problem is posed subject to the homogeneous or non-homogeneous Neumann boundary conditions, we should adapt our energy orthogonal decomposition to these conditions. For example, suppose on the boundary we have $\frac{\partial u}{\partial n}=g$,
  where $n$ is the unit-normal vector at the boundary and $g$ is a function on $\partial \Omega$. For an element $T$ adjacent to the boundary, we will decompose the solution $u=u^{\sfh}_{T}+u^{\sfb}_{T}+u^{\mathsf{p}}_T$ where
   \begin{equation}
    \label{eqn: harmonic-bubble splitting bc}
    \begin{aligned}
    &\left\{
    \begin{aligned}
    -\nabla \cdot (a \nabla u_T^\sfh )&=0, \quad \text{in} \  T\\
    u_T^\sfh&=u, \quad \text{on} \  \partial T \backslash \partial \Omega\\
    \frac{\partial u_T^\sfh}{\partial n}&=0, \quad \text{on} \  \partial T \cap \partial \Omega \, ,
    \end{aligned}
    \right.
    \\
    &\left\{
    \begin{aligned}
    -\nabla \cdot (a \nabla u^\sfb_T )&=f, \quad \text{in} \  T\\
    u^\sfb_T&=0, \quad \text{on} \  \partial T \backslash \partial \Omega \\
    \frac{\partial u_T^\sfb}{\partial n}&=0, \quad \text{on} \  \partial T \cap \partial \Omega \, .
    \end{aligned}
    \right.
    \end{aligned}
    \end{equation}
 The additional, third component $u^{\mathsf{p}}_T$ is to incorporate the non-homogeneous boundary condition using a particular solution:
    \begin{equation}
    \left\{
    \begin{aligned}
    -\nabla \cdot (a \nabla u^{\mathsf{p}}_T )&=0, \quad \text{in} \  T\\
    u^{\mathsf{p}}_T&=0, \quad \text{on} \  \partial T \backslash \partial \Omega \\
    \frac{\partial u_T^{\mathsf{p}}}{\partial n}&=g, \quad \text{on} \  \partial T \cap \partial \Omega \, .
    \end{aligned}
    \right.
    \end{equation}
Similar to the case of the homogeneous Dirichlet boundary condition, we can approximate the three components separately. Both $u^{\sfb}_{T}$ and $u^{\mathsf{p}}_T$ can be approximated via local computation, and the $a$-harmonic part $u^{\sfh}_{T}$ can be done via a similar oversampling and SVD.

The second natural generalization is to higher dimensions $d\geq 3$. In this case, the nodal interpolation part is not enough to localize the approximation onto each edge with co-dimension $1$. If we do not use the partition of unity and still use a non-overlapping domain decomposition, we may need to design approximation spaces for all cells with co-dimension $1,2,...,d$ that constitute the mesh. Indeed, the method in this paper for $d=2$ can be understood as a way of reducing the task of approximation of functions in $T$ of dimension $d$, to approximation of functions on edges (of co-dimension $1$) and nodes (of co-dimension $2$). The enrichment part is designed for the former case, and the interpolation part is for the latter case. For general $d$, this framework entails us to approximate functions on cells with co-dimension ranging from $1$ to $d$.

\subsection{Conclusions} Overall, this paper explores a novel energy orthogonal decomposition of the solution space, which takes advantage of the structural information of the different components, such as the locality in computation (for the bubble part) and the exponentially efficient approximation (for the $a$-harmonic part). This is why the overall exponential convergence can be achieved. Furthermore, in our non-overlapping domain decomposition, we build a general framework for edge coupling that can integrate local errors to a guarantee of global accuracy. It is of interest to extend this methodology beyond the case of ellipicity, namely to other types of PDEs, such as the Helmholtz equation, especially in the high frequency regime. We will explore these possibilities in our subsequent works.




\bibliographystyle{siamplain}

\end{document}